\documentclass[11pt,reqno]{amsart}
\usepackage{geometry}
\usepackage{amsmath,amsfonts,amsthm}
\usepackage{tikz-cd}
\usepackage{enumerate}
\usepackage{comment}

\theoremstyle{plain}

\newtheorem{theorem}{Theorem}[section]
\newtheorem{lemma}[theorem]{Lemma}

\theoremstyle{remark}

\newtheorem{remark*}{Remark}

\newcommand{\R}{\mathbb{R}}
\newcommand{\C}{\mathcal{C}}

\newcommand{\Sph}{\mathbb{S}}

\DeclareMathOperator{\tr}{trace}

\newcommand{\bmat}{\left(\begin{smallmatrix}}
	\newcommand{\emat}{\end{smallmatrix}\right)}

\newcommand{\into}{\hookrightarrow}
\newcommand{\SO}{\mathrm{SO}}

\newcommand{\norm}[1]{\left\lVert#1\right\rVert}

\usepackage{tikz,amsmath, amssymb,bm,color,pgfplots,pgfplotstable}
\usetikzlibrary{shapes,arrows}
\usetikzlibrary{calc}
\usetikzlibrary{fadings}
\usetikzlibrary{patterns}
\usetikzlibrary{shadows.blur}
\usetikzlibrary{shapes}

\makeatletter
\@namedef{subjclassname@2020}{%
  \textup{2020} Mathematics Subject Classification}
\makeatother

\title{F-Invariant Minimal Surfaces}

\numberwithin{equation}{section}
\allowdisplaybreaks

\begin{document}

\title{A family of free boundary minimal surfaces in the unit ball}

\author{Anna Siffert}
\address{Mathematisches Institut\\
Einsteinstr. 62\\
48149 M\" unster\\
Germany}
\email{ASiffert@uni-muenster.de}
\thanks{The first author gratefully acknowledges the
supports of the Deutsche Forschungsgemeinschaft (DFG, German Research Foundation) - Project-ID 427320536 - SFB 1442, as well as Germany's Excellence Strategy EXC 2044 390685587, Mathematics M\"unster: Dynamics-Geometry-Structure.
 Most of the work was done when the second named author wrote his Master thesis under the supervision of the first named author at the University of Bonn.}

\author{Jan Wuzyk}

\email{janwuzyk@gmail.com}


\begin{abstract}
Using equivariant differential geometry, we provide a family of free boundary minimal surfaces in the unit ball.
\end{abstract}

\maketitle

\section{Introduction}

Minimal surfaces are a classical field of study in geometry. They have been studied since 1760 \cite{Lagrange} and remain an active field of study. Their study began with the problem of finding a surface in $\R^3$ with minimal area given a fixed boundary.
Solutions to this problem have zero mean curvature. 
In what follows we consider any surface which has this property to be a minimal surface.
Inspite of how well studied minimal surfaces are, new examples are still hard are to come by as the partial differential equations (PDEs) describing minimal surfaces are generally intractable. That said examples have played a critical role in shaping research in the field, both by informing the creation of new conjectures and discrediting old. 

\smallskip

A particularly nice source of examples arises from methods from "Equivariant Differential Geometry". This is the study of differential geometric constructions which admit some additional symmetries. Usually the assumption of additional symmetries greatly simplifies the intractable equations describing certain differential geometric objects and allows specific examples to be constructed. This is exactly the case for the class of so-called $F$-invariant minimal surfaces which will be studied in this manuscript.
In 1971 Hsiang and Lawson \cite{hsiang} proved a reduction theorem reducing the problem of finding minimal surfaces in a manifold $M$ invariant under action of a Lie group $G$ to the much simpler problem of finding minimal surfaces in $M/G$. In \cite{WangIso}, a similar reduction theorem was proven for what we call $F$-invariant hypersurfaces. Continuing this work Wang classified all such minimal hypersurfaces in \cite{Wang}. Unfortunately, it seems this paper has not been widely read since, in the following years, many papers have attempted to classify certain simple subsets of $G$-invariant minimal surfaces, for example \cite{Alencar,AlencarAl,McGrath}. Most of the results in these papers are direct consequences of results proved in \cite{Wang}. Due to this we feel it would be helpful to write a more accessible account of (some of) Wang's results.

\smallskip

The notion of a $F$-invariant hypersurface in $\R^n$ is as follows: An isoparametric hypersurface $M$ in the sphere $\Sph^{n-1}$ is a hypersurface with constant principal curvatures. Say this has $g$ distinct principal curvatures with multiplicities $m_1$ and $m_2$.  All the parallel translates of $M$ are also isoparametric hypersurface except for two, the so called focal submanifolds $V_1$ and $V_2$, which are of codimension $m_1+1$ and $m_2+1$ in $\Sph^{n-1}$. There always exists a function $F$ on $\Sph^{n-1}$ such that the level sets of $F$ are exactly the isoparametric hypersurfaces. This partition of $\Sph^{n-1}$ can be extended to $\R^n$ by homotheties. We call a hypersurface composed of a union of leaves of this partition $F$-invariant. In any such family, in $\Sph^{n-1}$, there exists one isoparametric hypersurface which is minimal, say $M^*$. The cone over this, the minimal cone $C(M^*)$ is the simplest example of a $F$-invariant minimal hypersurface. Since isoparametric hypersurfaces generalise hypersurfaces in the sphere invariant under the action of a Lie group, $F$-invariant hypersurfaces generalise hypersurfaces invariant under a subgroup of $O(n)$. 

\smallskip

A subset of minimal surfaces which has attracted particular attention in recent years are those contained in the unit ball which intersect its boundary perpendicularly, called free boundary minimal surfaces. Freidin, Gulian and McGrath \cite{McGrath} already found many examples of free boundary minimal surfaces with some extra symmetry, but they did it in the context of \cite{AlencarAl}. This only covered $O(n)\times O(m)$ symmetries. Using the set up in \cite{Wang} we extend this to the much bigger class of all $F$-invariant minimal surfaces. In particular, we construct two families of $F$-invariant free boundary minimal surfaces. The family $\Sigma_{g,m_1,m_2}^k$ was already constructed in \cite{Wang}, but $\Omega_{g,m_1,m_2}$ is new. Again, see Section \ref{Isoparametric hypersurfaces} for a full list of possible triples. 

\begin{theorem}
	Given an isoparametric hypersurface in $S^{n-1}$ with corresponding triple $(g,m_1,m_2)$ where $n=\frac{m_1+m_2}{2}g+2$, we can construct F-invariant free boundary minimal surfaces in $\R^n$ in the following ways:
	\begin{enumerate}
		\item 	If $n<4g$, for all natural numbers $k$ we can construct a free boundary minimal surface $\Sigma_{g,m_1,m_2}^k$.
		\item If $(g,m_1,m_2) \neq (2,1,5),(2,5,1), (4,1,6), (4,6,1)$ and $n\geq 4g$, we can construct a free boundary minimal surface $\Omega_{g,m_1,m_2}$.
	\end{enumerate}
\end{theorem}

\textbf{Organisation:}
Preliminaries are provided in Section\,\ref{prelim}: we give a brief review on Wang's classification of $F$-invariant minimal surfaces in $\R^n$ since this will be relevant later when studying free boundary minimal surfaces in balls.
Afterwards we give some preliminaries on isoparametric hypersurfaces before we finally provide the reduction theorem.
In Section\,\ref{analysis}, we derive the vector field which describes the dynamical system we will be studying.
Further, in this section we begin the analysis of this dynamical system. 
In Section\,\ref{profile}, we use the classification of orbits to derive a classification of so called profile curves. 
Finally, in Section\,\ref{fbms}, we use our set up to construct  a new family of free boundary minimal surfaces in the unit ball. \\

\section{Preliminaries}
\label{prelim}
 In Subsection\,\ref{wang} we give a brief review on Wang's construction of $F$-invariant minimal surfaces in $\R^n$. Subsection\,\ref{Isoparametric hypersurfaces} provides preliminaries 
  on
isoparametric hypersurfaces. Subsection\,\ref{vectorcalculus} contains the proof of the reduction theorem. 

\subsection{Wang's construction of $F$-invariant minimal surfaces in $\R^n$}
\label{wang}
Wang's approach to classifying all $F$-invariant minimal surfaces in $\R^n$ works as follows: One can use well know techniques from equivariant differential geometry to show that it suffices to consider $F$-invariant variations to show that a $F$-invariant hypersurface is minimal. Given a fixed $F$ we can parametrise all the isoparametric hypersurfaces in $\Sph^{n-1}$ and their translates by two coordinates, $(r,\varphi)$ in a domain $D_g$. The problem of showing a $F$-invariant surface is minimal is thus equivalent to showing that its projection to the domain $D_g$, a so called profile curve, is a geodesic. We explicitly derive the metric on $D_g$ which depends just on the 3 numbers $(g,m_1,m_2)$. Using the invariance of geodesics in $D_g$ to homotheties, we reduce the problem of finding geodesics further to simply the analysis of a dynamical system in the plane. We can then use the well developed theory of dynamical systems to completely classify the behaviour of the orbits of this system, each of which correspond to a family of profile curves and further minimal surfaces. 
This allows us to answer questions about the embeddedness, asymptotic behaviour, topological type and stability of all $F$-invariant minimal surfaces. These are summarised in the following two theorems. 
A full list of triples satisfying these conditions is given in Section \ref{Isoparametric hypersurfaces}.

\begin{theorem}[Wang \cite{Wang}]
	\label{A}
	Any complete $F$-invariant minimal surfaces $M$ in $\R^n$, with triple $(g,m_1,m_2)$, such that $n<4g$, is of one of the following types:
	\begin{itemize}
		\item Cone : $M$ is the cone $C(M^*)$ where $M^*$ is the unique minimal isoparametric hypersurface in $\Sph^{n-1}$. Further, $M$ is unstable if $g>1$ and area minimizing if $g=1$ \cite{Wang}.
		\item Type I : $M$ is embedded and asymptotic to the cone $C(M^*)$, intersecting it infinitely often. Further, M is diffeomorphic to $\Sph^n-V_1$ or $\Sph^n-V_2$ and unstable.
		\item Type II : $M$ is immersed, intersecting itself infinitely often and asymptotic to the cone $C(M^*)$. Further, M is diffeomorphic to $\Sph^n-\{V_1\cup V_2\}$ and unstable.
	\end{itemize}
\end{theorem}

\begin{theorem}[Wang \cite{Wang}]
	\label{B}
	Any complete $F$-invariant minimal surfaces $M$ in $\R^n$, with triple $(g,m_1,m_2)$, such that $n\geq 4g$, is of one of the following types:
	
	\begin{itemize}
		\item Cone : $M$ is the cone $C(M^*)$. Further, $C(M^*)$ is stable. If $(g,m_1,m_2)\neq (2,1,5)$ or $(4,1,6)$, $C(M^*)$ is also area minimizing but otherwise not \cite{Wang}.	
		\item Type I :  M is embedded and asymptotic to the cone $C(M^*)$. Further, $M$ is diffeomorphic to $\Sph^n-V_1$ or $\Sph^n-V_2$. If $(g,m_1,m_2)\neq (2,1,5)$ or $(4,1,6)$, M is stable, otherwise not.	
		\item Type II : $M$ is embedded and asymptotic to the cone $C(M^*$, intersecting it at most once.  Further, M is diffeomorphic to $\Sph^n-\{V_1\cup V_2\}$ and is unstable.		
	\end{itemize}
\end{theorem}

\subsection{Brief review of isoparametric hypersurfaces}
\label{Isoparametric hypersurfaces}


In this subsection we discuss isoparametric hypersurfaces and some of their properties. More details can be found in \cite{Cecil,survey,Harmonic}.

\smallskip

Given a Riemannian manifold $M$, we call a smooth function $f: M\to \R$ an \textit{isoparametric function} if there exist smooth functions $T, S:\R\to\R$ such that:
\begin{equation}
	\norm{\nabla f}^2 = T\circ f, \qquad \Delta f = S\circ f.
	\label{dag}
\end{equation}
Given such a function $f$ we can decompose $M$ into a union of fibres of $f$. 
We call a fibres with maximal rank \textit{regular}. The subset of $M$ which consists of such fibres is denoted by $M^\circ$. The complement of $M^\circ$ in $M$ consists of the \textit{singular} fibres. The regular fibres are called \textit{isoparametric hypersurfaces} and the singular fibers are called \textit{focal submanifolds}. There is a particularly interesting theory of such functions in the case $M=\Sph^{n-1}$ and from now on we restrict ourselves to this case. Geometrically, (\ref{dag}) implies that isoparametric hypersurfaces in spheres have constant principal curvatures. In fact, any such hypersurface also gives rise to an isoparametric function. For an isoparametric hypersurface $M$, we denote the number of distinct principal curvatures by $g$.  In \cite{Munzner1}, M\"unzner proved that if we order the principal curvatures by value and let $m_i$ be the multiplicity of the $i$th principal curvature, then 
\[m_i=m_{i+2} \quad i \mod g.\]
Further, in \cite{Munzner2}, he showed that the only possible values for $g$ are $g\in\{1,2,3,4,6\}$. Directly from counting dimensions, we see that these numbers satisfy the identity 
\[\frac{m_1+m_2}{2}g=n-2\]
in $\Sph^{n-1}$. We can actually describe the geometry of a family of isoparametric hypersurfaces quite well. Let $f:\Sph^{n-1}\to \R$ be an isoparametric function and assume the fibres of $f$ are connected and $f(\Sph^{n-1})=[c_1,c_2]$. The focal submanifolds $f^{-1}(c_i)=V_i$, $i\in\{1,2\}$, are smooth minimal submanifolds of $\Sph^{n-1}$ of codimension $m_i+1$ and dist$(V_1,V_2)=\frac{\pi}{g}$. For any $c\in(c_1,c_2)$, $f^{-1}(c)$ is a smooth compact hypersurface which is the tube $M_{\varphi_i}$ of constant radius $\varphi_i$ around $V_i$, $\varphi_i\in(0,\frac{\pi}{g})$. The principal curvatures of $M_\varphi$ are $\cot(\varphi)$, $\cot(\varphi+\frac{\pi}{g})$, $\cot(\varphi+\frac{2\pi}{g})\dots$ with multiplicities $m_1,m_2,m_1 \dots$ . The volume of $M_\varphi$ is given by 
\[v(\varphi)=C\left(\sin\frac{g\varphi}{2}\right)^{m_1}\left(\cos\frac{g\varphi}{2}\right)^{m_2}\]
for some constant $C$ and the mean curvature is given by 
\[H(\varphi)= -\frac{d}{d\varphi}\log v(\varphi).\]

\smallskip

Examples for isoparametric hypersurfaces in spheres include the following:
\begin{itemize}
	\item For g = 1, the focal submanifolds are poles and the $M_\varphi$ are spheres parallel to the equator.
	\item For g = 2, the focal submanifolds are $\Sph^{m_1}$, $\Sph^{m_2}$ and the isoparametric hypersurfaces are then Clifford tori $\Sph^{m_1}(r)\times \Sph^{m_2}(s)$ where $r^2+s^2=1$.
	\item More generally, if $G$ is some Lie subgroup of $\SO(n)$ then the orbits of $G$ form a family of isoparametric hypersurfaces (under some restrictions).
	We call such a family \textit{homogeneous}. If $m_1=1$ or $m_2=1$ then this is the case. Many, but not all, isoparametric hypersurfaces are of this form. Homogeneous isoparametric hypersurfaces can be completely classified based on \cite{hsiang}.
	\item For $g=3$, Cartan \cite{Cartan39} proved that we have $m_1=m_2\in \{1,2,4,8\} $. Further, he also showed that the focal submanifolds are given by embeddings of $\mathbb{RP}^2$, $\mathbb{CP}^2$ or $\mathbb{HP}^2$ and the $M_\varphi$ are tubes around these.
	\item For $g=4$, there are two families of homogeneous isoparametric hypersurfaces. In \cite{FKM}, Ferus, Karcher and M\"unzner constructed examples of inhomogeneous isoparametric hypersurfaces based on representations of Clifford algebras with $g=4$. These are called FKM-type isoparametric hypersurfaces. The classification of isoparametric hypersurfaces with $g=4$ was completed by Chi in \cite{chi2020}. 
	\item For $g=6$, Abresch \cite{Abresch} proved that $m_1=m_2\in \{1,2\}$. In \cite{Dorfmeister}, Dorfmeister and Neher showed all examples with $m_1=1$ are homogeneous, but the classification is not complete for $m_1=2$ \cite{siffert}.
\end{itemize}
In summary, the following triples are possible. We list them with the conditions $n<4g$ and $n\geq 4g$ as these will be critical in what follows.

\medskip

\begin{center}
\begin{tabular}{ p{1.5cm}|p{6cm}|p{6cm} }
	& $n<4g$ & $n\geq 4g$\\
	\hline
	$g=1:$ & $(1,1,1)$ & $(1,m,m)$  for $m\geq 2$\\
	\hline
	$g=2:$ & $(2,m_1,m_2)$  for $m_1+m_2 <6 $& $(2,m_1,m_2)$  for $m_1+m_2\geq 6$\\
	\hline
	$g=3:$ & $(3,1,1)$, $(3,2,2)$ & $(3,4,4)$, $(3,8,8)$\\
	\hline
	$g=4:$ & $(4,1,1)$, $(4,1,2)$, $(4,1,3)$, $(4,1,4)$, $(4,1,5)$, $(4,2,2)$, $(4,2,3)$ & $(4,4,5)$ or $(4,m_1,m_2)$ for $m_1+m_2\geq 7$ and $m_1+m_2+1$ a multiple of $2^{\phi(m_1-1)}$, where $\phi(l)$ is the number of integers $s$ with $1 \leq s \leq l$ and $s \equiv 0, 1, 2, 4 \mod 8$, see Stolz \cite{Stolz}.\\
	\hline
	$g=6:$ & $(6,1,1)$, $(6,2,2)$ & \\
\end{tabular}
\end{center}


\subsection{Reduction theorem}
\label{vectorcalculus}

To prove the reduction theorem we will use a standard reduction technique from equivariant differential geometry. In this section, we introduce this technique and the notions from vector calculus needed to understand it. We finally provide the reduction theorem.
This section is based on \cite{EellsRatto}.

\smallskip

 We begin by fixing some notation. We let $(M^m,g)$ and $(N^n,h)$ be Riemannian manifolds of dimensions $m$ and $n$ respectively and let $\varphi:M\to N$ be a smooth map.
A minimal surface is a critical point of the volume functional. Precisely, we say \[\varphi:M\hookrightarrow N\] 
is minimal if $\varphi$ is a Riemannian immersion and $\varphi$ is an extremum of the volume functional
\[V(\varphi)= \int_M |\Lambda^m d\varphi|dx\]
with respect to variations $\varphi_t$ through Riemannian immersions. 
Note that we endow $M$ with the induced metric $\varphi^*h$ and thus $|\Lambda^m d\varphi|dx$ is the associated volume form.
The \textit{tension field} of $\varphi$ is defined to be the trace of the second fundamental form
\[\tau(\varphi)=\tr \nabla d\varphi\in \varphi^{-1}TN.\]
Let $\varphi_t$ be a variation through Riemannian immersions, then we have
\[\frac{d V(\varphi_t)}{dt}_{\lvert t=0} = -\int_M \left\langle\tau(\varphi) , \frac{\partial\varphi_t}{\partial t}_{\lvert t=0}\right\rangle \nu,\]
where $\nu$ is the volume element induced by $\varphi_0^*h=\varphi^*h$. As a corollary of this, we have that $M$ is minimal if and only if 
\[\tau(\varphi)=0.\] 
To state the reduction theorem, we have to define a few more notions. 
We call a map $F:M\to N$ \textit{transnormal} if there exists a function, $T: N\to \R$, such that 
\[\norm{\nabla F}^2 = T\circ F.\]
The fibres of such a map have constant rank and we call the set on which they have maximal rank the \textit{regular} set $M^0\subset M$. Further, there exists a unique \textit{quotient metric} $h_F$ on $N^0=F(M^0)$, such that $F:(M^0,g)\to (N^0,h_F)$ is a Riemannian submersion. In fact, if two transnormal maps have the same fibres, then the quotient metrics are the same. We call a map $F:M\to N$ \textit{isoparametric} if it is transnormal and additionally there exists a vector field $S\in \C(TN)$, such that 
\[\tau(F)=S\circ F.\]
Note that this definition is in accordance with that of Subsection\,\ref{Isoparametric hypersurfaces}, i.e. for $N=\R$ we recover the definition of Subsection\,\ref{Isoparametric hypersurfaces}.
Finally, we introduce the notion $(\rho,\sigma)$-equivariance. Let $\rho:M\to P$, $\sigma:N\to Q$ be Riemannian submersions. We say a map $\varphi:M\to N$ is \textit{$(\rho, \sigma)$-equivariant} if there exits $\bar{\varphi}$ such that the following diagram commutes: 
\begin{equation*}
\begin{tikzcd}
M \arrow{r}{\varphi} \arrow{d}{\rho} & N \arrow{d}{\sigma}\\
P \arrow{r}{\bar{\varphi}} & Q
\end{tikzcd}.
\end{equation*}

\begin{theorem}[ Theorem 4.5, \cite{EellsRatto}]
\label{Thm:EllisRato}
Let $\varphi$ be a $(\rho,\sigma)$-equivariant map. Assume that
\begin{itemize}
	\item[(a)] $\sigma,\rho$ are isoparametric,
	\item[(b)] $d\varphi(T^HM)\subset T^HN$,
	\item[(c)] $\varphi:\rho^{-1}(z)\to \sigma^{-1}(\bar{\varphi}(z))$ is a Riemannian submersion for all $z \in P$ , with respect to the induced metrics,
\end{itemize}
then $\tau(\varphi)=0$ iff it is stationary with respect to $(\rho,\sigma)$-equivariant variations.
\end{theorem}


With this preparation at hand we can finally provide the reduction theorem.
Note that a similar reduction theorem has already been stated in \cite{WangIso}.
Since we could not follow the proof in \cite{wang2}, we provide one here.

\begin{theorem}[The reduction theorem]
	Let $f: M \to N$ be an isoparametric map with connected compact fibres and image $N$. Let $N^o$ be the regular part of $N$ and $N_1 \subset N^o$, a submanifold. $M_1=f^{-1}(N_1)$ is minimal
	in $M$ if and only if $N_1$ is minimal in $N^o$, with respect to the following metric
	\[ds^2=v^2ds_f^2,\]
	where $v$ is the volume function of the fibres of $f$ and $ds_f^2$ is the quotient metric.
	\label{Reduction}
\end{theorem}

\begin{proof}
	We can view $M_1$ be the pullback of $F$ over $N_1$ 
	\[M_1 = \{(x,y)\in \R^n\times N:\varphi(x)=\tilde{F}(y),\tilde{F}(y)\in N_1\}.\]
	Then we have the following commutative diagram:
	\begin{equation*}
	\begin{tikzcd}
	M_1 \arrow{r}{\varphi} \arrow{d}{f|_{M_1}} & M \arrow{d}{f}\\
	N_1 \arrow{r}{\bar{\varphi}} & N
	\end{tikzcd}
	\end{equation*}

	where $\varphi$ is the inclusion map. We can apply Theorem \ref{Thm:EllisRato} since:
	\begin{itemize}
		\item[(a)] We immediately have that $f$ and $f|_M$ are isoparametric.\\
		\item[(b)] $d\varphi(T^HM)\subset T^HN$ is immediate since $\varphi$ is an inclusion and the  $f|_{M_1}$ is a restriction of $f$. \\
		\item[(c)] Follows immediately since $\varphi$ is an inclusion.
	\end{itemize}
	We conclude that $M_1$ is minimal if and only if it is stationary with respect to $f$-invariant variations, i.e. variations through fibres of $f$. If we let $v(x):\R^n\to\R$ denote the volume of a fibre then the volume of $M_1$ equals the volume of $N_1$ with respect to the metric 
	\[v^2ds_{f^2}\]
	on $N^o$. Thus  $\varphi:M_1\into M$ is minimal exactly when $N_1$ is.
\end{proof}

\section{Derivation and Analysis of the Vector Field}
\label{analysis}
In Subsection\,\ref{derivation}, we reduce the problem of finding $F$-invariant hypersurface in $\R^n$ to the study of a system of ordinary differential equations. 
In Subsection\,\ref{analysissub} we deduce the behaviour of 
the vector field which describes the dynamical system we will be studying.
We will begin in an essentially standard way. We determine the singular points of the vector field.
We then discuss the local behaviour at these points. Next, we analyse the global behaviour. A clever argument allows us to show the system has no periodic orbits. This allows us to determine the long term behaviour of all orbits. An especially important question is the behaviour of the separatrices of the vector field. 
All results stated in this section can be found in \cite{Wang}. We added those proofs which we felt might be helpful for the reader.

\subsection{Derivation of equations}
\label{derivation}
We now use the reduction theorem to make the problem tractable.

\smallskip

Recall, given a family of isoparametric hypersurfaces $\mathcal{F}$ in $\Sph^{n-1}$, a $F$-invariant hypersurface in $\R^n$ is a surface, which is a union of leaves all isometric to a dilation of hypersurface in $\mathcal{F}$. Recall that to any such $\mathcal{F}$ there is a Cartan-M\"unzner Polynomial $F$, such that $\mathcal{F}$ are the level sets of $F$ in $\Sph^{n-1}$. We would like to use this to construct a isoparametric map whose leaves are dilations of isoparametric hypersurfaces. We define the map $\tilde{F}:\R^n\to \R^2$ by
\[\widetilde{F}(x)=(F(x),\norm{x}^2).\]
That this is isoparametric, follows from the properties of Cartan-M\"unzner polynomials, see Subsection \ref{Isoparametric hypersurfaces}. We then apply the reduction theorem and conclude that 
a $F$-invariant hypersurface $M_1$ is minimal if and only if $\Gamma=\tilde{F}(M_1)$ is a geodesic in $D_g=\{re^{i\varphi}:0\leq r < \infty, 0\leq \varphi \leq \frac{\pi}{g}\}$ with the metric
\[h=(r^{n-2}v(\varphi))^2(dr^2+r^2d\varphi^2).\]

Straightforward considerations, which can be found in \cite{Wang}, show that
this problem is equivalent to providing solutions of the system
\begin{equation}
\begin{cases}
\frac{d\theta}{dt}=\sin\alpha\sin2\theta\\
\frac{d\alpha}{dt} = -m \sin\alpha\sin2\theta + 2\cos\alpha(m_1\cos^2\theta-m_2\sin^2\theta)
\end{cases}.
\label{sysfinal}
\end{equation}
This defines a vector field on the strip $[0,\frac{\pi}{2}]\times(-\infty,\infty)$.

\subsection{Analysis of the vector field}
\label{analysissub}

\begin{figure}
\begin{center}

\tikzset{every picture/.style={line width=0.75pt}} 

\begin{tikzpicture}[x=0.75pt,y=0.75pt,yscale=-1,xscale=1,scale = 0.8]

\draw    (149.34,0) -- (149.34,123) -- (150,300) ;
\draw    (510,0) -- (510,300) ;
\draw  [dash pattern={on 0.84pt off 2.51pt}]  (149.34,190) .. controls (242.28,212.13) and (246.26,211.14) .. (288.05,240) .. controls (329.84,268.86) and (400.41,269.13) .. (510,290) ;
\draw [shift={(510,290)}, rotate = 10.78] [color={rgb, 255:red, 0; green, 0; blue, 0 }  ][fill={rgb, 255:red, 0; green, 0; blue, 0 }  ][line width=0.75]      (0, 0) circle [x radius= 3.35, y radius= 3.35]   ;
\draw [shift={(149.34,190)}, rotate = 13.39] [color={rgb, 255:red, 0; green, 0; blue, 0 }  ][fill={rgb, 255:red, 0; green, 0; blue, 0 }  ][line width=0.75]      (0, 0) circle [x radius= 3.35, y radius= 3.35]   ;
\draw  [dash pattern={on 0.84pt off 2.51pt}]  (149.34,140) -- (510,140) ;
\draw [shift={(510,140)}, rotate = 0] [color={rgb, 255:red, 0; green, 0; blue, 0 }  ][fill={rgb, 255:red, 0; green, 0; blue, 0 }  ][line width=0.75]      (0, 0) circle [x radius= 3.35, y radius= 3.35]   ;
\draw [shift={(149.34,140)}, rotate = 0] [color={rgb, 255:red, 0; green, 0; blue, 0 }  ][fill={rgb, 255:red, 0; green, 0; blue, 0 }  ][line width=0.75]      (0, 0) circle [x radius= 3.35, y radius= 3.35]   ;
\draw  [dash pattern={on 0.84pt off 2.51pt}]  (149.34,90) .. controls (239.5,111.13) and (246.26,111.14) .. (288.05,140) .. controls (329.84,168.86) and (397.64,171.13) .. (510,190) ;
\draw [shift={(510,190)}, rotate = 9.54] [color={rgb, 255:red, 0; green, 0; blue, 0 }  ][fill={rgb, 255:red, 0; green, 0; blue, 0 }  ][line width=0.75]      (0, 0) circle [x radius= 3.69, y radius= 3.69]   ;
\draw [shift={(149.34,90)}, rotate = 13.19] [color={rgb, 255:red, 0; green, 0; blue, 0 }  ][fill={rgb, 255:red, 0; green, 0; blue, 0 }  ][line width=0.75]      (0, 0) circle [x radius= 3.69, y radius= 3.69]   ;
\draw  [dash pattern={on 0.84pt off 2.51pt}]  (149.34,-10) .. controls (237.42,-9.44) and (246.26,11.14) .. (288.05,40) .. controls (329.84,68.86) and (383.77,71.13) .. (510,90) ;
\draw  [dash pattern={on 0.84pt off 2.51pt}]  (149.34,240) -- (510,240) ;
\draw [shift={(510,240)}, rotate = 0] [color={rgb, 255:red, 0; green, 0; blue, 0 }  ][fill={rgb, 255:red, 0; green, 0; blue, 0 }  ][line width=0.75]      (0, 0) circle [x radius= 3.35, y radius= 3.35]   ;
\draw [shift={(149.34,240)}, rotate = 0] [color={rgb, 255:red, 0; green, 0; blue, 0 }  ][fill={rgb, 255:red, 0; green, 0; blue, 0 }  ][line width=0.75]      (0, 0) circle [x radius= 3.35, y radius= 3.35]   ;
\draw    (204.82,150) -- (204.82,132) ;
\draw [shift={(204.82,130)}, rotate = 450] [color={rgb, 255:red, 0; green, 0; blue, 0 }  ][line width=0.75]    (10.93,-3.29) .. controls (6.95,-1.4) and (3.31,-0.3) .. (0,0) .. controls (3.31,0.3) and (6.95,1.4) .. (10.93,3.29)   ;
\draw    (426.77,250) -- (426.77,232) ;
\draw [shift={(426.77,230)}, rotate = 450] [color={rgb, 255:red, 0; green, 0; blue, 0 }  ][line width=0.75]    (10.93,-3.29) .. controls (6.95,-1.4) and (3.31,-0.3) .. (0,0) .. controls (3.31,0.3) and (6.95,1.4) .. (10.93,3.29)   ;
\draw    (260.31,10) -- (206.82,10) ;
\draw [shift={(204.82,10)}, rotate = 360] [color={rgb, 255:red, 0; green, 0; blue, 0 }  ][line width=0.75]    (10.93,-3.29) .. controls (6.95,-1.4) and (3.31,-0.3) .. (0,0) .. controls (3.31,0.3) and (6.95,1.4) .. (10.93,3.29)   ;
\draw    (215.92,110) -- (269.41,110) ;
\draw [shift={(271.41,110)}, rotate = 180] [color={rgb, 255:red, 0; green, 0; blue, 0 }  ][line width=0.75]    (10.93,-3.29) .. controls (6.95,-1.4) and (3.31,-0.3) .. (0,0) .. controls (3.31,0.3) and (6.95,1.4) .. (10.93,3.29)   ;
\draw    (260.31,210) -- (206.82,210) ;
\draw [shift={(204.82,210)}, rotate = 360] [color={rgb, 255:red, 0; green, 0; blue, 0 }  ][line width=0.75]    (10.93,-3.29) .. controls (6.95,-1.4) and (3.31,-0.3) .. (0,0) .. controls (3.31,0.3) and (6.95,1.4) .. (10.93,3.29)   ;
\draw  [dash pattern={on 0.84pt off 2.51pt}]  (149.34,40) -- (510,40) ;
\draw [shift={(510,40)}, rotate = 0] [color={rgb, 255:red, 0; green, 0; blue, 0 }  ][fill={rgb, 255:red, 0; green, 0; blue, 0 }  ][line width=0.75]      (0, 0) circle [x radius= 3.35, y radius= 3.35]   ;
\draw [shift={(149.34,40)}, rotate = 0] [color={rgb, 255:red, 0; green, 0; blue, 0 }  ][fill={rgb, 255:red, 0; green, 0; blue, 0 }  ][line width=0.75]      (0, 0) circle [x radius= 3.35, y radius= 3.35]   ;
\draw    (204.82,30) -- (204.82,48) ;
\draw [shift={(204.82,50)}, rotate = 270] [color={rgb, 255:red, 0; green, 0; blue, 0 }  ][line width=0.75]    (10.93,-3.29) .. controls (6.95,-1.4) and (3.31,-0.3) .. (0,0) .. controls (3.31,0.3) and (6.95,1.4) .. (10.93,3.29)   ;
\draw    (426.77,130) -- (426.77,148) ;
\draw [shift={(426.77,150)}, rotate = 270] [color={rgb, 255:red, 0; green, 0; blue, 0 }  ][line width=0.75]    (10.93,-3.29) .. controls (6.95,-1.4) and (3.31,-0.3) .. (0,0) .. controls (3.31,0.3) and (6.95,1.4) .. (10.93,3.29)   ;
\draw    (426.77,50) -- (426.77,32) ;
\draw [shift={(426.77,30)}, rotate = 450] [color={rgb, 255:red, 0; green, 0; blue, 0 }  ][line width=0.75]    (10.93,-3.29) .. controls (6.95,-1.4) and (3.31,-0.3) .. (0,0) .. controls (3.31,0.3) and (6.95,1.4) .. (10.93,3.29)   ;
\draw    (204.82,230) -- (204.82,248) ;
\draw [shift={(204.82,250)}, rotate = 270] [color={rgb, 255:red, 0; green, 0; blue, 0 }  ][line width=0.75]    (10.93,-3.29) .. controls (6.95,-1.4) and (3.31,-0.3) .. (0,0) .. controls (3.31,0.3) and (6.95,1.4) .. (10.93,3.29)   ;
\draw    (410.12,170) -- (356.64,170) ;
\draw [shift={(354.64,170)}, rotate = 360] [color={rgb, 255:red, 0; green, 0; blue, 0 }  ][line width=0.75]    (10.93,-3.29) .. controls (6.95,-1.4) and (3.31,-0.3) .. (0,0) .. controls (3.31,0.3) and (6.95,1.4) .. (10.93,3.29)   ;
\draw    (354.64,270) -- (408.12,270) ;
\draw [shift={(410.12,270)}, rotate = 180] [color={rgb, 255:red, 0; green, 0; blue, 0 }  ][line width=0.75]    (10.93,-3.29) .. controls (6.95,-1.4) and (3.31,-0.3) .. (0,0) .. controls (3.31,0.3) and (6.95,1.4) .. (10.93,3.29)   ;
\draw    (354.64,70) -- (408.12,70) ;
\draw [shift={(410.12,70)}, rotate = 180] [color={rgb, 255:red, 0; green, 0; blue, 0 }  ][line width=0.75]    (10.93,-3.29) .. controls (6.95,-1.4) and (3.31,-0.3) .. (0,0) .. controls (3.31,0.3) and (6.95,1.4) .. (10.93,3.29)   ;
\draw    (149.34,123) ;
\draw    (288.05,140) ;
\draw [shift={(288.05,140)}, rotate = 0] [color={rgb, 255:red, 0; green, 0; blue, 0 }  ][fill={rgb, 255:red, 0; green, 0; blue, 0 }  ][line width=0.75]      (0, 0) circle [x radius= 3.35, y radius= 3.35]   ;
\draw    (288.05,40) ;
\draw [shift={(288.05,40)}, rotate = 0] [color={rgb, 255:red, 0; green, 0; blue, 0 }  ][fill={rgb, 255:red, 0; green, 0; blue, 0 }  ][line width=0.75]      (0, 0) circle [x radius= 3.35, y radius= 3.35]   ;
\draw    (288.05,240) ;
\draw [shift={(288.05,240)}, rotate = 0] [color={rgb, 255:red, 0; green, 0; blue, 0 }  ][fill={rgb, 255:red, 0; green, 0; blue, 0 }  ][line width=0.75]      (0, 0) circle [x radius= 3.35, y radius= 3.35]   ;
\draw    (150,290) ;
\draw [shift={(150,290)}, rotate = 0] [color={rgb, 255:red, 0; green, 0; blue, 0 }  ][fill={rgb, 255:red, 0; green, 0; blue, 0 }  ][line width=0.75]      (0, 0) circle [x radius= 3.35, y radius= 3.35]   ;

\draw (106.38,132) node [anchor=north west][inner sep=0.75pt]   [align=left] {0};
\draw (127,10) node [anchor=north west][inner sep=0.75pt]   [align=left] {$\displaystyle \alpha $};
\draw (487,142) node [anchor=north west][inner sep=0.75pt]   [align=left] {$\displaystyle \theta $};
\draw (91,82) node [anchor=north west][inner sep=0.75pt]   [align=left] {$\displaystyle \pi /2$};
\draw (106,32) node [anchor=north west][inner sep=0.75pt]   [align=left] {$\displaystyle \pi $};
\draw (75,180) node [anchor=north west][inner sep=0.75pt]   [align=left] {$\displaystyle -\pi /2$};
\draw (96,230) node [anchor=north west][inner sep=0.75pt]   [align=left] {$\displaystyle -\pi $};
\draw (251.31,82) node [anchor=north west][inner sep=0.75pt]  [font=\small] [align=left] {$\displaystyle \eta _{\alpha }$};
\draw (362.28,22) node [anchor=north west][inner sep=0.75pt]  [font=\small] [align=left] {$\displaystyle \eta _{\theta }$};
\draw (362.28,222) node [anchor=north west][inner sep=0.75pt]  [font=\small] [align=left] {$\displaystyle \eta _{\theta }$};
\draw (251.31,182) node [anchor=north west][inner sep=0.75pt]  [font=\small] [align=left] {$\displaystyle \eta _{\alpha }$};
\draw (362.28,122) node [anchor=north west][inner sep=0.75pt]  [font=\small] [align=left] {$\displaystyle \eta _{\theta }$};
\draw (445.51,62) node [anchor=north west][inner sep=0.75pt]  [font=\small] [align=left] {$\displaystyle \eta _{\alpha }$};
\draw (168.97,72) node [anchor=north west][inner sep=0.75pt]   [align=left] {$\displaystyle S_{1}$};
\draw (471.37,160) node [anchor=north west][inner sep=0.75pt]   [align=left] {$\displaystyle S_{3}$};
\draw (168.97,170) node [anchor=north west][inner sep=0.75pt]   [align=left] {$\displaystyle S_{2}$};
\draw (471.37,262) node [anchor=north west][inner sep=0.75pt]   [align=left] {$\displaystyle S_{4}$};
\draw (293.59,220) node [anchor=north west][inner sep=0.75pt]   [align=left] {$\displaystyle O_{+}$};
\draw (295.47,120) node [anchor=north west][inner sep=0.75pt]   [align=left] {$\displaystyle O_{-}$};
\draw (71,280) node [anchor=north west][inner sep=0.75pt]   [align=left] {$\displaystyle -3\pi /2$};

\end{tikzpicture}
	
	\caption{Zero sets of the vector field.}
	\label{Fig:Zero}
	
\end{center}	
\end{figure}
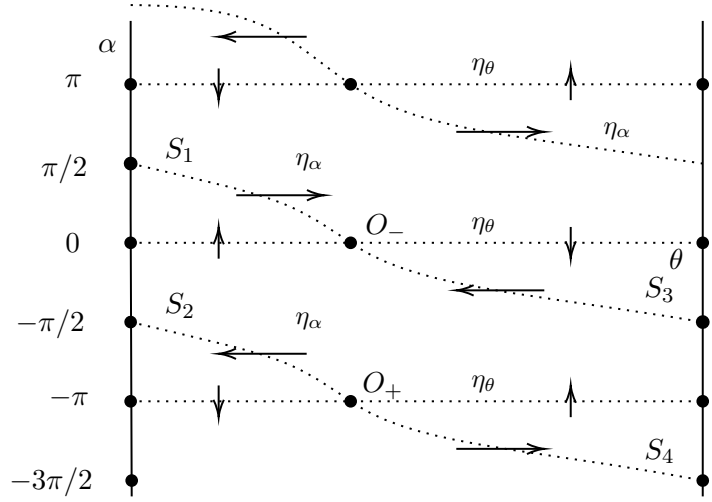

Let $\eta_\theta$ (resp. $\eta_\alpha$) denote the set along which $\frac{d\theta}{dt}=0$ (resp. $\frac{d\alpha}{dt}=0$). We then have 
\[\eta_\theta = \{(0,\alpha)\}\cup\{(\frac{\pi}{2},\alpha)\}\cup\{(\theta,k\pi)| k \in \mathbb{Z}\},\]
\[\eta_\alpha = \{(\theta,\alpha)\quad | \quad 0 = -m \sin\alpha\sin2\theta + 2\cos\alpha(m_1\cos^2\theta-m_2\sin^2\theta)\}.\]
These sets intersect at the points $(0,\frac{\pi}{2}+k\pi)$, $(\theta^*,k\pi)$ and $(\frac{\pi}{2},\frac{\pi}{2} + k\pi)$ , $k\in\mathbb{Z}$, where 
\[\theta^*= \arctan\sqrt{\frac{m_1}{m_2}}.\] 

The vector field is invariant under the transformation $(\theta,\alpha)\to (\theta,\alpha+2\pi)$, so it suffices to describe the local behaviour around representatives for each class of fixed points in $[0,\frac{\pi}{2}]\times[-3\frac{\pi}{2},\frac{\pi}{2}]$. This gives us the following points:
\[
S_1: (0,\frac{\pi}{2})	\quad				
S_2: (0,-\frac{\pi}{2})		\quad			
S_3: (\frac{\pi}{2},-\frac{\pi}{2})\quad
S_4: (\frac{\pi}{2},-\frac{3\pi}{2})\quad
O_-: (\theta^*,0)\quad
O_+: (\theta^*,-\pi),
\]
compare Figure\,\ref{Fig:Zero}.

\begin{lemma}[Local picture at $S_i$]
	\label{Si}
	The $S_i$ are saddle points. Further
	\begin{enumerate}[1{)}]
		\item 	At $S_1$, the stable manifold is the vertical segment $0\times(-\frac{\pi}{2},\frac{3\pi}{2})$.
		\item For small $\theta\in(-\epsilon,\epsilon)$, the unstable manifold $\gamma_1$ is given by $\alpha(\theta)$, such that:
		\[\alpha(0)=\frac{\pi}{2},\quad \alpha'(0)=-\frac{m}{m_1+1}.\]
	\end{enumerate}
  $S_3$ has analogous stable and unstable manifolds. $S_2$ and $S_4$ are also analogous with unstable and stable manifolds switched.
\end{lemma}

\begin{lemma}[Local picture at $O_i$]
	$O_+$ is a source and $O_-$ is a sink. The points $O_i$ are focal if $n<4g$ and nodal if $n \geq 4g$.
	Further:
	\begin{enumerate}[1{)}]
		\item If $n<4g$, any curve $\gamma$ entering (leaving) $O_-$ ($O_+$) intersects the horizontal axis through $O_i$ infinitely many times.
		\item If $n\geq 4g$, all except 2 curves $\gamma$ enter $O_-$ with a slope $\lambda_-$. The two exceptional curves enter $O_-$ with a slope $\mu_-$. 
		Here $\lambda_-=\frac{1}{2}(-m+\sqrt{\Delta})$ and $\mu_-=\frac{1}{2}(-m-\sqrt{\Delta})$, where
		$m = m_1+m_2 + g/2$ and $\Delta = m^2-8(m_1+m_2)$.
	\end{enumerate} 
\end{lemma}

Additionally, solution curves of the system have the following properties. These will help determine their long term behaviour.
\begin{lemma} Let $\gamma(t)=(\theta(t),\alpha(t))$ be a non-singular orbit with $0<\theta < \frac{\pi}{2}$.
	\begin{enumerate}[1{)}]
		\item Whenever $\gamma(t_0) \in \eta_{\theta }$ then $\theta(t_0)$ is a local minimum (resp. maximum) as a function of $\alpha$ if $\theta(t_0)<\theta^*$ (resp. $\theta(t_0)>\theta^*$).
	
		\item Whenever $\gamma(t_0) \in \eta_{\alpha }$ then $\alpha(t_0)$ is a local maximum (resp. minimum) as a function of $\theta$ if $\theta(t_0)<\theta^*$ (resp. $\theta(t_0)>\theta^*$).
	\end{enumerate}
	As a consequence, minima and maxima of $\theta$ (resp. $\alpha$) occur alternately along $\gamma$.
			\label{minmax}
\end{lemma}
 Now we will prove that the vector field has no periodic orbits. This is important since it limits the possible behaviour of the system and will allow us to characterise the global behaviour of the system. We will use the monotonicity formula for minimal surfaces \cite[Theorem 6.2]{Minimal} to arrive at a contradiction. 
\begin{theorem}[No periodic orbits]
\label{noperiodic}
	The vector field has no periodic orbits.
	\label{periodic}
\end{theorem}
\begin{proof}
	Let $\gamma=(\theta(t),\alpha(t))$ be a periodic orbit. Fixing some initial $r$ value we associate to it a geodesic $\Gamma=(r(s),\varphi(s),\alpha(s))$ in $D_g$. Here, $r(s)$ solves 
	\[\frac{d(\log(r))}{dt_1}=\cos(\alpha(t_1)), \frac{dt_1}{ds}=\frac{1}{r} \text{  and  } \theta = \frac{g}{2}\varphi.\]
	For any constant $h$, the homothety \[h_*:(r,\varphi)\mapsto(hr,\varphi)\]
	preserves geodesics setwise in $D_g$. Since $\gamma$ is periodic, \[(r(T),\varphi(T),\alpha(T))=(\frac{r(T)}{r(0)} r(0),\varphi(0),\alpha(0))\]
	for some $T$. By the unqiueness of solutions and the homothety property, if $\delta=\frac{r(T)}{r(0)}>1$, we must have $\Gamma = \delta_*\Gamma$ as these are both solution curves containing the same point. 
	Let $M_\Gamma$ be a minimal surface in $\R^n$ such that $\tilde{F}(M) = \Gamma$, then $M_\Gamma$ is $\delta$-self-similar. We will apply the monotonicity formula to this surface. To show $\delta>1$, first define
	\[\beta = \int_0^T\cos(\alpha(t_1))dt_1=\int_{0}^{T}\frac{d (\log r)}{dt_1}dt_1=\log\left(\frac{r(T)}{r(0)}\right).\]
	Then
	\[ e^{\beta}=\frac{r(T)}{r(0)}=\delta.\]
	
	Now, we show some constraints on $\gamma$ which prove $\beta>0$. We know the index of $\gamma$ must be $1$. It can not go around any saddle points as the vertical lines $\theta=0$ and $\theta = \frac{\pi}{2}$ are invariant sets. The only remaining fixed points are sinks and saddles on the line $\theta = \theta^*$. We know $\gamma$ can only encircle one of them so we assume w.l.o.g that $\gamma$ encircles $O_-$. Let $\alpha_+$ and $\alpha_-$ be the maximum and minimum of $\alpha(t)$. By Lemma \ref{minmax}, at $\alpha_+$ we must have $k\pi<\alpha<k\pi+\frac{\pi}{2}$. If $\alpha_+>\pi$, it would encircle another fixed point, thus $\alpha_+ <\frac{\pi}{2}$. Similarly, $\alpha_->-\frac{\pi}{2}$. Thus $\frac{\pi}{2}<\alpha<\frac{\pi}{2}$, so $\cos\alpha(t)>0$, consequently $\beta > 0$ and thus that $\delta>1$.\\
	
	Recall that the monotonicity formula for minimal surfaces \cite[Theorem 6.2]{Minimal} states that
	for a minimal surface $M$, 
	\[\frac{\mbox{Area}(M\cap B_\rho)}{\rho^{n-1}}\]
	is increasing. For a $\delta$-self-similar minimal surface, we know $$\frac{\mbox{Area}(M\cap B_\rho)}{\rho^{n-1}} = \frac{\mbox{Area}(M\cap B_{\delta\rho)}}{(\delta\rho)^{n-1}},$$ so this fraction must be constant. Thus, we have
	\[\mbox{Area}(M_\Gamma\cap B_\rho)=c\rho^{n-1}\]
	for some constant $c$.
	More explicitly, parametrise $\Gamma(s)= (r(s),\varphi(s))$ by Euclidean arc length as in (2.6). Let $s_0(\rho)$ be the solution to $r(s)=\rho$, i.e $s_0=r^{-1}(\rho)$. Now we can rewrite our equation as
	\[c\rho^{n-1}=\int_0^{s_0(\rho)}r^{n-2}(s)v(\varphi(s))ds.\]
	Taking the $\rho$ derivative on both sides gives us
	\[(n-1)c\rho^{n-2}=\frac{ds_0}{d\rho}r^{n-2}(s_0)v(\varphi(s_0)).\]
	Cancelling the $\rho$s and applying the inverse function theorem, we have
	\[(n-1)c= \frac{1}{\frac{dr}{ds}(s_0)}v(\varphi(s_0).\]
	By (2.6) and rearranging, we get
	\[\cos(\alpha(s_0))(n-1)c=v(\varphi(s_0)).\]
	
	Since $\gamma$ encircles $O_-$, it must intersect the $\theta$-axis at some point, say $\theta(t_0) < \theta^*$. We then have
	\[\left.\frac{d^2}{d\alpha^2}\log\cos(\alpha)\right|_{\alpha(t_0)}=-1\]
	but, 
	\[\left.\frac{d^2}{d\alpha^2}\log v(\varphi(\alpha))\right|_{\alpha(t_0)}= H(\varphi(\alpha_0))\varphi''(\alpha_0)>0,\]
	since $\theta_0<\theta^*$ and $\theta_0$ is a minimum. This gives a contradiction.\\
\end{proof}

This result allows us to describe the the global behaviour of orbits. We denote by $\gamma_i$ the non-vertical unstable/stable manifold of $S_i$. These will be separatrices.
\begin{lemma}[Behaviour of the separatrices]
The curves $\gamma_i$ are seperatrices, in particular:
\begin{align*}
 &L^+(\gamma_1)=O_-\quad,\quad L^-(\gamma_1) = S_1,\\
 &L^+(\gamma_2)=S_2\quad,\quad L^-(\gamma_1) = O_+,\\
 &L^+(\gamma_3)=O_-\quad,\quad L^-(\gamma_1) = S_3,\\
 &L^+(\gamma_4)=S_4\quad,\quad L^-(\gamma_1) = O_+.\\
\end{align*}

\end{lemma}
\begin{figure}[!htb]
\label{curvi}
\begin{center}	
	
	\tikzset{every picture/.style={line width=0.75pt}} 
	
	\begin{tikzpicture}[x=0.75pt,y=0.75pt,yscale=-1,xscale=1,scale = 0.8]
	
	\draw [line width=1.5]    (210,30) -- (210,230) ;
	\draw [line width=1.5]    (610,40) -- (610,240) ;
	\draw [line width=1.5]  [dash pattern={on 1.69pt off 2.76pt}]  (209,138) -- (610,140) ;
	\draw [shift={(610,140)}, rotate = 0.29] [color={rgb, 255:red, 0; green, 0; blue, 0 }  ][fill={rgb, 255:red, 0; green, 0; blue, 0 }  ][line width=1.5]      (0, 0) circle [x radius= 4.36, y radius= 4.36]   ;
	\draw [shift={(209,138)}, rotate = 0.29] [color={rgb, 255:red, 0; green, 0; blue, 0 }  ][fill={rgb, 255:red, 0; green, 0; blue, 0 }  ][line width=1.5]      (0, 0) circle [x radius= 4.36, y radius= 4.36]   ;
	\draw [line width=1.5]  [dash pattern={on 1.69pt off 2.76pt}]  (208.35,63) .. controls (263.5,80) and (333,111) .. (365.33,138) .. controls (397.66,165) and (484.87,201.13) .. (610,220) ;
	\draw [shift={(610,220)}, rotate = 8.58] [color={rgb, 255:red, 0; green, 0; blue, 0 }  ][fill={rgb, 255:red, 0; green, 0; blue, 0 }  ][line width=1.5]      (0, 0) circle [x radius= 4.79, y radius= 4.79]   ;
	\draw [shift={(208.35,63)}, rotate = 17.13] [color={rgb, 255:red, 0; green, 0; blue, 0 }  ][fill={rgb, 255:red, 0; green, 0; blue, 0 }  ][line width=1.5]      (0, 0) circle [x radius= 4.79, y radius= 4.79]   ;
	\draw [line width=1.5]    (208.35,96) ;
	\draw [line width=1.5]    (370,140) ;
	\draw [shift={(370,140)}, rotate = 0] [color={rgb, 255:red, 0; green, 0; blue, 0 }  ][fill={rgb, 255:red, 0; green, 0; blue, 0 }  ][line width=1.5]      (0, 0) circle [x radius= 4.36, y radius= 4.36]   ;
	\draw    (208.35,63) .. controls (321.5,63) and (390.5,114) .. (400,140) ;
	\draw [shift={(400,140)}, rotate = 69.93] [color={rgb, 255:red, 0; green, 0; blue, 0 }  ][fill={rgb, 255:red, 0; green, 0; blue, 0 }  ][line width=0.75]      (0, 0) circle [x radius= 3.35, y radius= 3.35]   ;
	\draw    (400,140) -- (400,190) ;
	\draw [shift={(400,190)}, rotate = 90] [color={rgb, 255:red, 0; green, 0; blue, 0 }  ][fill={rgb, 255:red, 0; green, 0; blue, 0 }  ][line width=0.75]      (0, 0) circle [x radius= 3.35, y radius= 3.35]   ;
	\draw    (310,140) .. controls (311.5,173) and (361.5,196) .. (610,220) ;
	\draw    (310,80) -- (310,140) ;
	\draw [shift={(310,140)}, rotate = 90] [color={rgb, 255:red, 0; green, 0; blue, 0 }  ][fill={rgb, 255:red, 0; green, 0; blue, 0 }  ][line width=0.75]      (0, 0) circle [x radius= 3.35, y radius= 3.35]   ;
	\draw [shift={(310,80)}, rotate = 90] [color={rgb, 255:red, 0; green, 0; blue, 0 }  ][fill={rgb, 255:red, 0; green, 0; blue, 0 }  ][line width=0.75]      (0, 0) circle [x radius= 3.35, y radius= 3.35]   ;
	\draw [line width=1.5]  [dash pattern={on 1.69pt off 2.76pt}]  (210,220) ;
	\draw [shift={(210,220)}, rotate = 0] [color={rgb, 255:red, 0; green, 0; blue, 0 }  ][fill={rgb, 255:red, 0; green, 0; blue, 0 }  ][line width=1.5]      (0, 0) circle [x radius= 4.36, y radius= 4.36]   ;
	\draw [shift={(210,220)}, rotate = 0] [color={rgb, 255:red, 0; green, 0; blue, 0 }  ][fill={rgb, 255:red, 0; green, 0; blue, 0 }  ][line width=1.5]      (0, 0) circle [x radius= 4.36, y radius= 4.36]   ;
	
	\draw (162.14,134) node [anchor=north west][inner sep=0.75pt]   [align=left] {0};
	\draw (585.07,115) node [anchor=north west][inner sep=0.75pt]   [align=left] {$\displaystyle \theta $};
	\draw (145.26,55) node [anchor=north west][inner sep=0.75pt]   [align=left] {$\displaystyle \pi /2$};
	\draw (145,210) node [anchor=north west][inner sep=0.75pt]   [align=left] {$\displaystyle -\pi /2$};
	\draw (252,82) node [anchor=north west][inner sep=0.75pt]  [font=\small] [align=left] {$\displaystyle \eta _{1}$};
	\draw (211,43) node [anchor=north west][inner sep=0.75pt]   [align=left] {$\displaystyle S_{1}$};
	\draw (591,192) node [anchor=north west][inner sep=0.75pt]   [align=left] {$\displaystyle S_{3}$};
	\draw (365,120) node [anchor=north west][inner sep=0.75pt]   [align=left] {$\displaystyle O_{-}$};
	\draw (401,120) node [anchor=north west][inner sep=0.75pt]   [align=left] {$\displaystyle A_{1}$};
	\draw (256.12,48) node [anchor=north west][inner sep=0.75pt]   [align=left] {$\displaystyle \gamma _{1}$};
	\draw (402,193) node [anchor=north west][inner sep=0.75pt]   [align=left] {$\displaystyle B_{1}$};
	\draw (308.5,71) node [anchor=south west] [inner sep=0.75pt]   [align=left] {$\displaystyle B_{2}$};
	\draw (287,140) node [anchor=north west][inner sep=0.75pt]   [align=left] {$\displaystyle A_{2}$};
	\draw (480,212) node [anchor=north west][inner sep=0.75pt]   [align=left] {$\displaystyle \gamma _{ \begin{array}{{>{\displaystyle}l}}
			3\\
			\end{array}}$};

	\end{tikzpicture}
	\caption[Lemma 7]{The curvilinear rectangle $A_1B_1A_2B_2$.}
\end{center}
\end{figure}
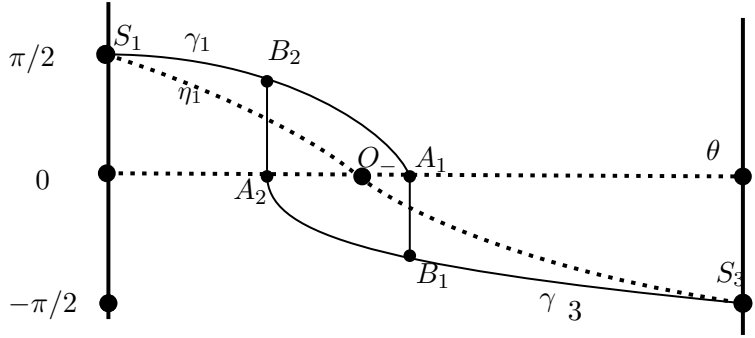

\begin{proof} We prove this for $\gamma_1=(\theta_\gamma(t),\alpha_\gamma(t))$, the other proofs are analogous.
	We denote by $\xi$ the portion of $\eta_{\alpha }$ connecting $S_1$ and $O_-$ . This can be written as a function of $\theta$ for $0<\theta<\frac{\pi}{2}$ as
	\[\alpha_\xi(\theta)=\arctan(\frac{h(\theta)}{m}).\]
	In the limit, we have	
	\[ \alpha_\xi'(0)=-\frac{m}{m_1}.\]
	From \ref{Si} we have that 
	\[\alpha_\gamma(0)= \frac{-m}{m_1+1}> \alpha_\xi(0).\]
	Thus $\gamma_1$ lies above $\xi$ initially. Further, $\alpha_\gamma$ is decreasing and $\theta_\gamma$ is decreasing. Recall that $\alpha$ has a local maximum for any curve crossing $\xi$. Thus $\gamma$ can not cross it before crossing the $\theta$ axis. We now have 2 cases:\\ 
	
	1) If $\gamma_1$ never crosses the $\theta$ axis then $\alpha$ decreases montonically and $\theta$ increases monotonically, thus $L^+(\gamma_1)=O_-$ is the limit point.\\
	
	2) If $\gamma_1$ crosses the $\theta$ axis, call this point $A_1$, draw a verticle line down to $\gamma_3$ and call this $B_1$. Repeat the same argument for $\gamma_3$ and call the intersection $A_2$ and $B_2$. This gives a curvilinear rectangle, see Figure\,\ref{curvi}. $\gamma_1$ can not exit this rectangle and the only limit point inside this rectangle is $O_-$. Since our system has no periodic orbits we must have $L^+(\gamma_1)=O_-$ by the Poincare-Bendixon theorem.
\end{proof}

\begin{figure}[!h]
\label{fundamental}
	\begin{center}

		
		\tikzset{
			pattern size/.store in=\mcSize, 
			pattern size = 5pt,
			pattern thickness/.store in=\mcThickness, 
			pattern thickness = 0.3pt,
			pattern radius/.store in=\mcRadius, 
			pattern radius = 1pt}
		\makeatletter
		\pgfutil@ifundefined{pgf@pattern@name@_pml4lgoy8}{
			\pgfdeclarepatternformonly[\mcThickness,\mcSize]{_pml4lgoy8}
			{\pgfqpoint{0pt}{0pt}}
			{\pgfpoint{\mcSize+\mcThickness}{\mcSize+\mcThickness}}
			{\pgfpoint{\mcSize}{\mcSize}}
			{
				\pgfsetcolor{\tikz@pattern@color}
				\pgfsetlinewidth{\mcThickness}
				\pgfpathmoveto{\pgfqpoint{0pt}{0pt}}
				\pgfpathlineto{\pgfpoint{\mcSize+\mcThickness}{\mcSize+\mcThickness}}
				\pgfusepath{stroke}
		}}
		\makeatother
		\tikzset{every picture/.style={line width=0.75pt}} 
		
		\begin{tikzpicture}[x=0.75pt,y=0.75pt,yscale=-1,xscale=1,scale=0.8]
		
		\draw [line width=1.5]    (175.3,31) -- (175.3,114) -- (172.6,281) ;
		\draw [line width=1.5]    (586.6,31) -- (586.6,281) ;
		\draw [line width=1.5]  [dash pattern={on 1.69pt off 2.76pt}]  (175.3,131) -- (586.6,131) ;
		\draw [shift={(175.3,131)}, rotate = 0] [color={rgb, 255:red, 0; green, 0; blue, 0 }  ][fill={rgb, 255:red, 0; green, 0; blue, 0 }  ][line width=1.5]      (0, 0) circle [x radius= 4.36, y radius= 4.36]   ;
		\draw [line width=1.5]  [dash pattern={on 1.69pt off 2.76pt}]  (175.3,231) -- (586.6,231) ;
		\draw [line width=1.5]  [dash pattern={on 1.69pt off 2.76pt}]  (175.3,31) -- (586.6,31) ;
		\draw [shift={(586.6,31)}, rotate = 0] [color={rgb, 255:red, 0; green, 0; blue, 0 }  ][fill={rgb, 255:red, 0; green, 0; blue, 0 }  ][line width=1.5]      (0, 0) circle [x radius= 4.36, y radius= 4.36]   ;
		\draw [shift={(175.3,31)}, rotate = 0] [color={rgb, 255:red, 0; green, 0; blue, 0 }  ][fill={rgb, 255:red, 0; green, 0; blue, 0 }  ][line width=1.5]      (0, 0) circle [x radius= 4.36, y radius= 4.36]   ;
		\draw [line width=1.5]    (175.3,114) ;
		\draw [line width=1.5]    (333.49,131) ;
		\draw [shift={(333.49,131)}, rotate = 0] [color={rgb, 255:red, 0; green, 0; blue, 0 }  ][fill={rgb, 255:red, 0; green, 0; blue, 0 }  ][line width=1.5]      (0, 0) circle [x radius= 4.36, y radius= 4.36]   ;
		\draw [line width=1.5]    (333.49,231) ;
		\draw [shift={(333.49,231)}, rotate = 0] [color={rgb, 255:red, 0; green, 0; blue, 0 }  ][fill={rgb, 255:red, 0; green, 0; blue, 0 }  ][line width=1.5]      (0, 0) circle [x radius= 4.36, y radius= 4.36]   ;
		\draw [line width=1.5]    (174.33,181.33) ;
		\draw [shift={(174.33,181.33)}, rotate = 0] [color={rgb, 255:red, 0; green, 0; blue, 0 }  ][fill={rgb, 255:red, 0; green, 0; blue, 0 }  ][line width=1.5]      (0, 0) circle [x radius= 4.36, y radius= 4.36]   ;
		\draw  [pattern=_pml4lgoy8,pattern size=37.5pt,pattern thickness=0.75pt,pattern radius=0pt, pattern color={rgb, 255:red, 0; green, 0; blue, 0}][line width=1.5]  (174.82,81.22) .. controls (174.47,68.97) and (313,108.8) .. (335.71,131.22) .. controls (358.42,153.64) and (587.4,186.4) .. (587.31,181.22) .. controls (587.22,176.04) and (586.76,257.26) .. (586.63,276.86) .. controls (586.5,296.47) and (586.6,281.58) .. (586.6,281) .. controls (586.6,280.42) and (366.2,257.22) .. (334.2,231.22) .. controls (302.21,205.22) and (173.86,193.83) .. (174.33,181.33) .. controls (174.81,168.83) and (174.51,155.97) .. (174.94,140.22) .. controls (175.37,124.47) and (175.3,123.72) .. (175.3,114) .. controls (175.3,104.28) and (175.16,93.47) .. (174.82,81.22) -- cycle ;
		\draw [line width=1.5]    (586.6,281) ;
		\draw [shift={(586.6,281)}, rotate = 0] [color={rgb, 255:red, 0; green, 0; blue, 0 }  ][fill={rgb, 255:red, 0; green, 0; blue, 0 }  ][line width=1.5]      (0, 0) circle [x radius= 4.36, y radius= 4.36]   ;
		\draw [line width=1.5]    (174.82,81.22) ;
		\draw [shift={(174.82,81.22)}, rotate = 0] [color={rgb, 255:red, 0; green, 0; blue, 0 }  ][fill={rgb, 255:red, 0; green, 0; blue, 0 }  ][line width=1.5]      (0, 0) circle [x radius= 4.36, y radius= 4.36]   ;
		\draw [line width=1.5]    (587.31,181.22) ;
		\draw [shift={(587.31,181.22)}, rotate = 0] [color={rgb, 255:red, 0; green, 0; blue, 0 }  ][fill={rgb, 255:red, 0; green, 0; blue, 0 }  ][line width=1.5]      (0, 0) circle [x radius= 4.36, y radius= 4.36]   ;
		
		\draw (123,123) node [anchor=north west][inner sep=0.75pt]   [align=left] {0};
		\draw (143.2,3) node [anchor=north west][inner sep=0.75pt]   [align=left] {$\displaystyle \alpha $};
		\draw (539.47,133) node [anchor=north west][inner sep=0.75pt]   [align=left] {$\displaystyle \theta $};
		\draw (101,73) node [anchor=north west][inner sep=0.75pt]   [align=left] {$\displaystyle \pi /2$};
		\draw (121,23) node [anchor=north west][inner sep=0.75pt]   [align=left] {$\displaystyle \pi $};
		\draw (90,171) node [anchor=north west][inner sep=0.75pt]   [align=left] {$\displaystyle -\pi /2$};
		\draw (111,221) node [anchor=north west][inner sep=0.75pt]   [align=left] {$\displaystyle -\pi $};
		\draw (155.36,62.33) node [anchor=north west][inner sep=0.75pt]   [align=left] {$\displaystyle S_{1}$};
		\draw (558.88,156) node [anchor=north west][inner sep=0.75pt]   [align=left] {$\displaystyle S_{3}$};
		\draw (154.02,163.33) node [anchor=north west][inner sep=0.75pt]   [align=left] {$\displaystyle S_{2}$};
		\draw (566.88,257.33) node [anchor=north west][inner sep=0.75pt]   [align=left] {$\displaystyle S_{4}$};
		\draw (341.56,211) node [anchor=north west][inner sep=0.75pt]   [align=left] {$\displaystyle O_{+}$};
		\draw (288.2,83) node [anchor=north west][inner sep=0.75pt]   [align=left] {$\displaystyle \gamma _{1}$};
		\draw (490,143) node [anchor=north west][inner sep=0.75pt]   [align=left] {$\displaystyle \gamma _{2}$};
		\draw (234,177) node [anchor=north west][inner sep=0.75pt]   [align=left] {$\displaystyle \gamma _{3}$};
		\draw (474.87,269) node [anchor=north west][inner sep=0.75pt]   [align=left] {$\displaystyle \gamma _{4}$};
		\draw (82,271) node [anchor=north west][inner sep=0.75pt]   [align=left] {$\displaystyle -3\pi /2$};
		\draw (331.67,165.33) node [anchor=north west][inner sep=0.75pt]  [font=\LARGE] [align=left] {$\displaystyle G$\\};

		\end{tikzpicture}
		\caption{The Fundamental Region $G$.}
	\end{center}
\end{figure}
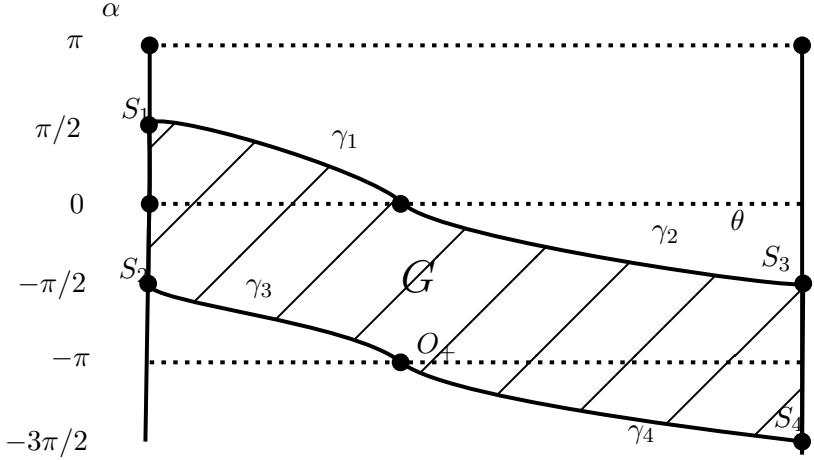
We can now define a fundamental domain $G$, bounded by the curves $\gamma_1$,$\gamma_2$,$\gamma_3$,$\gamma_4$ and the $\theta = 0$ and $\theta$=$\frac{\pi}{2}$ axis, see Figure\,\ref{fundamental}. We call this a fundamental domain since any orbit in the strip $0<\theta<\frac{\pi}{2}$ is a translate of an orbit in this domain by a translation $\alpha \to \alpha+k\pi$, for some $k\in \mathbb{Z}$, up to a change of direction. This justifies limiting our analysis to just this region.

\begin{lemma}[Behaviour of other curves]
	For any curve $\gamma$ not in boundary of $G$, we have 
	\[L^+(\gamma)=O_-,\quad L^-(\gamma)=O_+.\]
\end{lemma}

We call a triple $(g,m_1,m_2)$ with $n\geq 4g$ \textit{unstable} if the curves $\gamma_1$ or $\gamma_2$ cross the line $\theta=\theta^*$. Otherwise we call the triple \textit{stable}. We classify unstable and stable triples in the next section but for now we classify all the orbit of vector fields given by stable triples.

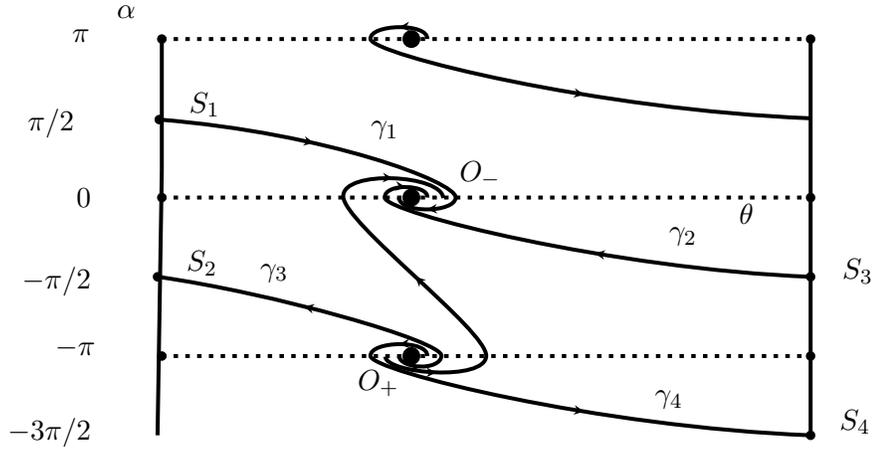
\begin{figure}[!htb]
	\begin{center}

		\tikzset{every picture/.style={line width=0.75pt}} 
		
		\begin{tikzpicture}[x=0.75pt,y=0.75pt,yscale=-1,xscale=1,scale = 0.8]
		
		\draw [line width=1.5]    (180,30) -- (180,113) -- (177.31,280) ;
		\draw [line width=1.5]    (589.02,30) -- (589.02,280) ;
		\draw [line width=1.5]  [dash pattern={on 1.69pt off 2.76pt}]  (180,130) -- (589.02,130) ;
		\draw [shift={(589.02,130)}, rotate = 0] [color={rgb, 255:red, 0; green, 0; blue, 0 }  ][fill={rgb, 255:red, 0; green, 0; blue, 0 }  ][line width=1.5]      (0, 0) circle [x radius= 1.74, y radius= 1.74]   ;
		\draw [shift={(180,130)}, rotate = 0] [color={rgb, 255:red, 0; green, 0; blue, 0 }  ][fill={rgb, 255:red, 0; green, 0; blue, 0 }  ][line width=1.5]      (0, 0) circle [x radius= 1.74, y radius= 1.74]   ;
		\draw [line width=1.5]  [dash pattern={on 1.69pt off 2.76pt}]  (180,230) -- (589.02,230) ;
		\draw [shift={(589.02,230)}, rotate = 0] [color={rgb, 255:red, 0; green, 0; blue, 0 }  ][fill={rgb, 255:red, 0; green, 0; blue, 0 }  ][line width=1.5]      (0, 0) circle [x radius= 1.74, y radius= 1.74]   ;
		\draw [shift={(180,230)}, rotate = 0] [color={rgb, 255:red, 0; green, 0; blue, 0 }  ][fill={rgb, 255:red, 0; green, 0; blue, 0 }  ][line width=1.5]      (0, 0) circle [x radius= 1.74, y radius= 1.74]   ;
		\draw [line width=1.5]  [dash pattern={on 1.69pt off 2.76pt}]  (180,30) -- (589.02,30) ;
		\draw [shift={(589.02,30)}, rotate = 0] [color={rgb, 255:red, 0; green, 0; blue, 0 }  ][fill={rgb, 255:red, 0; green, 0; blue, 0 }  ][line width=1.5]      (0, 0) circle [x radius= 1.74, y radius= 1.74]   ;
		\draw [shift={(180,30)}, rotate = 0] [color={rgb, 255:red, 0; green, 0; blue, 0 }  ][fill={rgb, 255:red, 0; green, 0; blue, 0 }  ][line width=1.5]      (0, 0) circle [x radius= 1.74, y radius= 1.74]   ;
		\draw [line width=1.5]    (180,113) ;
		\draw [line width=1.5]    (337.32,130) ;
		\draw [shift={(337.32,130)}, rotate = 0] [color={rgb, 255:red, 0; green, 0; blue, 0 }  ][fill={rgb, 255:red, 0; green, 0; blue, 0 }  ][line width=1.5]      (0, 0) circle [x radius= 4.36, y radius= 4.36]   ;
		\draw [line width=1.5]    (337.32,30) ;
		\draw [shift={(337.32,30)}, rotate = 0] [color={rgb, 255:red, 0; green, 0; blue, 0 }  ][fill={rgb, 255:red, 0; green, 0; blue, 0 }  ][line width=1.5]      (0, 0) circle [x radius= 4.36, y radius= 4.36]   ;
		\draw [line width=1.5]    (337.32,230) ;
		\draw [shift={(337.32,230)}, rotate = 0] [color={rgb, 255:red, 0; green, 0; blue, 0 }  ][fill={rgb, 255:red, 0; green, 0; blue, 0 }  ][line width=1.5]      (0, 0) circle [x radius= 4.36, y radius= 4.36]   ;
		\draw [line width=1.5]    (357.39,130.4) .. controls (356.91,115.5) and (294.26,112.17) .. (294.38,129.2) .. controls (294.5,146.23) and (385.26,213.83) .. (384.51,230) .. controls (383.76,246.17) and (321.71,241.17) .. (320.87,230.4) ;
		\draw [shift={(325.66,117.83)}, rotate = 183.81] [fill={rgb, 255:red, 0; green, 0; blue, 0 }  ][line width=0.08]  [draw opacity=0] (6.43,-3.09) -- (0,0) -- (6.43,3.09) -- (4.27,0) -- cycle    ;
		\draw [shift={(339.77,179.93)}, rotate = 42.97] [fill={rgb, 255:red, 0; green, 0; blue, 0 }  ][line width=0.08]  [draw opacity=0] (6.43,-3.09) -- (0,0) -- (6.43,3.09) -- (4.27,0) -- cycle    ;
		\draw [shift={(353.2,240.32)}, rotate = 180.21] [fill={rgb, 255:red, 0; green, 0; blue, 0 }  ][line width=0.08]  [draw opacity=0] (6.43,-3.09) -- (0,0) -- (6.43,3.09) -- (4.27,0) -- cycle    ;
		\draw [line width=1.5]    (589.02,280) .. controls (455.37,274.83) and (312.16,239.5) .. (311.56,230) .. controls (310.97,220.5) and (347.96,219.5) .. (347.37,230) ;
		\draw [shift={(446.25,265.01)}, rotate = 189.15] [fill={rgb, 255:red, 0; green, 0; blue, 0 }  ][line width=0.08]  [draw opacity=0] (6.43,-3.09) -- (0,0) -- (6.43,3.09) -- (4.27,0) -- cycle    ;
		\draw [shift={(329.74,222.5)}, rotate = 356.05] [fill={rgb, 255:red, 0; green, 0; blue, 0 }  ][line width=0.08]  [draw opacity=0] (6.43,-3.09) -- (0,0) -- (6.43,3.09) -- (4.27,0) -- cycle    ;
		\draw [shift={(589.02,280)}, rotate = 182.21] [color={rgb, 255:red, 0; green, 0; blue, 0 }  ][fill={rgb, 255:red, 0; green, 0; blue, 0 }  ][line width=1.5]      (0, 0) circle [x radius= 1.74, y radius= 1.74]   ;
		\draw [line width=1.5]    (178.38,80.8) .. controls (284.71,89.5) and (365.74,120.23) .. (365.27,130) .. controls (364.79,139.77) and (330.06,140.17) .. (329.46,130) ;
		\draw [shift={(275.86,95.63)}, rotate = 192.47] [fill={rgb, 255:red, 0; green, 0; blue, 0 }  ][line width=0.08]  [draw opacity=0] (6.43,-3.09) -- (0,0) -- (6.43,3.09) -- (4.27,0) -- cycle    ;
		\draw [shift={(347.15,137.48)}, rotate = 356.24] [fill={rgb, 255:red, 0; green, 0; blue, 0 }  ][line width=0.08]  [draw opacity=0] (6.43,-3.09) -- (0,0) -- (6.43,3.09) -- (4.27,0) -- cycle    ;
		\draw [shift={(178.38,80.8)}, rotate = 4.68] [color={rgb, 255:red, 0; green, 0; blue, 0 }  ][fill={rgb, 255:red, 0; green, 0; blue, 0 }  ][line width=1.5]      (0, 0) circle [x radius= 1.74, y radius= 1.74]   ;
		\draw [line width=1.5]    (177.31,180) .. controls (249.51,190.17) and (357.51,220.5) .. (356.32,230) .. controls (355.12,239.5) and (330.66,238.17) .. (329.46,230) ;
		\draw [shift={(269.52,198.66)}, rotate = 13.97] [fill={rgb, 255:red, 0; green, 0; blue, 0 }  ][line width=0.08]  [draw opacity=0] (6.43,-3.09) -- (0,0) -- (6.43,3.09) -- (4.27,0) -- cycle    ;
		\draw [shift={(342.89,236.63)}, rotate = 174.88] [fill={rgb, 255:red, 0; green, 0; blue, 0 }  ][line width=0.08]  [draw opacity=0] (6.43,-3.09) -- (0,0) -- (6.43,3.09) -- (4.27,0) -- cycle    ;
		\draw [shift={(177.31,180)}, rotate = 8.02] [color={rgb, 255:red, 0; green, 0; blue, 0 }  ][fill={rgb, 255:red, 0; green, 0; blue, 0 }  ][line width=1.5]      (0, 0) circle [x radius= 1.74, y radius= 1.74]   ;
		\draw [line width=1.5]    (347.37,130) .. controls (348.56,121.83) and (320.51,120.5) .. (320.51,130) .. controls (320.51,139.5) and (464.91,175.17) .. (589.02,180) ;
		\draw [shift={(589.02,180)}, rotate = 2.23] [color={rgb, 255:red, 0; green, 0; blue, 0 }  ][fill={rgb, 255:red, 0; green, 0; blue, 0 }  ][line width=1.5]      (0, 0) circle [x radius= 1.74, y radius= 1.74]   ;
		\draw [shift={(333.77,123.37)}, rotate = 185.87] [fill={rgb, 255:red, 0; green, 0; blue, 0 }  ][line width=0.08]  [draw opacity=0] (6.43,-3.09) -- (0,0) -- (6.43,3.09) -- (4.27,0) -- cycle    ;
		\draw [shift={(453.02,165.16)}, rotate = 9.7] [fill={rgb, 255:red, 0; green, 0; blue, 0 }  ][line width=0.08]  [draw opacity=0] (6.43,-3.09) -- (0,0) -- (6.43,3.09) -- (4.27,0) -- cycle    ;
		\draw [line width=1.5]    (589.02,80) .. controls (455.37,74.83) and (312.16,39.5) .. (311.56,30) .. controls (310.97,20.5) and (347.96,19.5) .. (347.37,30) ;
		\draw [shift={(446.25,65.01)}, rotate = 189.15] [fill={rgb, 255:red, 0; green, 0; blue, 0 }  ][line width=0.08]  [draw opacity=0] (6.43,-3.09) -- (0,0) -- (6.43,3.09) -- (4.27,0) -- cycle    ;
		\draw [shift={(329.74,22.5)}, rotate = 356.05] [fill={rgb, 255:red, 0; green, 0; blue, 0 }  ][line width=0.08]  [draw opacity=0] (6.43,-3.09) -- (0,0) -- (6.43,3.09) -- (4.27,0) -- cycle    ;
		
		\draw (124.37,122) node [anchor=north west][inner sep=0.75pt]   [align=left] {0};
		\draw (149.83,7) node [anchor=north west][inner sep=0.75pt]   [align=left] {$\displaystyle \alpha $};
		\draw (542.12,132) node [anchor=north west][inner sep=0.75pt]   [align=left] {$\displaystyle \theta $};
		\draw (93.68,72) node [anchor=north west][inner sep=0.75pt]   [align=left] {$\displaystyle \pi /2$};
		\draw (121.58,22) node [anchor=north west][inner sep=0.75pt]   [align=left] {$\displaystyle \pi $};
		\draw (90,172) node [anchor=north west][inner sep=0.75pt]   [align=left] {$\displaystyle -\pi /2$};
		\draw (111,217) node [anchor=north west][inner sep=0.75pt]   [align=left] {$\displaystyle -\pi $};
		\draw (195.56,62) node [anchor=north west][inner sep=0.75pt]   [align=left] {$\displaystyle S_{1}$};
		\draw (607.27,167) node [anchor=north west][inner sep=0.75pt]   [align=left] {$\displaystyle S_{3}$};
		\draw (193.77,162) node [anchor=north west][inner sep=0.75pt]   [align=left] {$\displaystyle S_{2}$};
		\draw (605.48,262) node [anchor=north west][inner sep=0.75pt]   [align=left] {$\displaystyle S_{4}$};
		\draw (301.75,237) node [anchor=north west][inner sep=0.75pt]   [align=left] {$\displaystyle O_{+}$};
		\draw (365.8,105) node [anchor=north west][inner sep=0.75pt]   [align=left] {$\displaystyle O_{-}$};
		\draw (310.12,80) node [anchor=north west][inner sep=0.75pt]   [align=left] {$\displaystyle \gamma _{1}$};
		\draw (498.08,147) node [anchor=north west][inner sep=0.75pt]   [align=left] {$\displaystyle \gamma _{2}$};
		\draw (240.31,170) node [anchor=north west][inner sep=0.75pt]   [align=left] {$\displaystyle \gamma _{3}$};
		\draw (489.13,250) node [anchor=north west][inner sep=0.75pt]   [align=left] {$\displaystyle \gamma _{4}$};
		\draw (81,267) node [anchor=north west][inner sep=0.75pt]   [align=left] {$\displaystyle -3\pi /2$};

		\end{tikzpicture}
		\caption{Typical orbits when $n<4g$.}	
	\end{center}	
\end{figure}

\begin{theorem}[Orbit Types A]
	\label{orbitA}
If $n<4g$ then any orbit $\gamma$ of the vector field in the domain $G$ with $0<\theta<\frac{\pi}{2}$ is of one of the following types:
\begin{itemize}
	\item Fixed  : $\gamma$ is one of the fixed points $O_+$ or $O_-$.
	\item Type I : A separatrix with $L^{\pm}(\gamma) = O_{\mp}$ and $L^{\mp}(\gamma) = S_i$, spiralling around $O_-$ and $O_+$.
	\item Type II : A curve with $L^+({\gamma})=O_-$ and $L^-({\gamma})=O_+$, spiralling into both.
\end{itemize}
\end{theorem}

\begin{figure}

	\begin{center}

	\tikzset{every picture/.style={line width=0.75pt}} 
	
	\begin{tikzpicture}[x=0.75pt,y=0.75pt,yscale=-1,xscale=1,scale = 0.8]
	
	\draw [line width=1.5]    (180.3,40) -- (180.3,123) -- (177.6,290) ;
	\draw [line width=1.5]    (591.6,40) -- (591.6,290) ;
	\draw [line width=1.5]  [dash pattern={on 1.69pt off 2.76pt}]  (180.3,140) -- (591.6,140) ;
	\draw [shift={(591.6,140)}, rotate = 0] [color={rgb, 255:red, 0; green, 0; blue, 0 }  ][fill={rgb, 255:red, 0; green, 0; blue, 0 }  ][line width=1.5]      (0, 0) circle [x radius= 4.36, y radius= 4.36]   ;
	\draw [shift={(180.3,140)}, rotate = 0] [color={rgb, 255:red, 0; green, 0; blue, 0 }  ][fill={rgb, 255:red, 0; green, 0; blue, 0 }  ][line width=1.5]      (0, 0) circle [x radius= 4.36, y radius= 4.36]   ;
	\draw [line width=1.5]  [dash pattern={on 1.69pt off 2.76pt}]  (180.3,240) -- (591.6,240) ;
	\draw [shift={(591.6,240)}, rotate = 0] [color={rgb, 255:red, 0; green, 0; blue, 0 }  ][fill={rgb, 255:red, 0; green, 0; blue, 0 }  ][line width=1.5]      (0, 0) circle [x radius= 4.36, y radius= 4.36]   ;
	\draw [shift={(180.3,240)}, rotate = 0] [color={rgb, 255:red, 0; green, 0; blue, 0 }  ][fill={rgb, 255:red, 0; green, 0; blue, 0 }  ][line width=1.5]      (0, 0) circle [x radius= 4.36, y radius= 4.36]   ;
	\draw [line width=1.5]  [dash pattern={on 1.69pt off 2.76pt}]  (180.3,40) -- (591.6,40) ;
	\draw [shift={(591.6,40)}, rotate = 0] [color={rgb, 255:red, 0; green, 0; blue, 0 }  ][fill={rgb, 255:red, 0; green, 0; blue, 0 }  ][line width=1.5]      (0, 0) circle [x radius= 4.36, y radius= 4.36]   ;
	\draw [shift={(180.3,40)}, rotate = 0] [color={rgb, 255:red, 0; green, 0; blue, 0 }  ][fill={rgb, 255:red, 0; green, 0; blue, 0 }  ][line width=1.5]      (0, 0) circle [x radius= 4.36, y radius= 4.36]   ;
	\draw [line width=1.5]    (180.3,123) ;
	\draw [line width=1.5]    (338.49,140) ;
	\draw [shift={(338.49,140)}, rotate = 0] [color={rgb, 255:red, 0; green, 0; blue, 0 }  ][fill={rgb, 255:red, 0; green, 0; blue, 0 }  ][line width=1.5]      (0, 0) circle [x radius= 4.36, y radius= 4.36]   ;
	\draw [line width=1.5]    (338.49,40) ;
	\draw [shift={(338.49,40)}, rotate = 0] [color={rgb, 255:red, 0; green, 0; blue, 0 }  ][fill={rgb, 255:red, 0; green, 0; blue, 0 }  ][line width=1.5]      (0, 0) circle [x radius= 4.36, y radius= 4.36]   ;
	\draw [line width=1.5]    (338.49,240) ;
	\draw [shift={(338.49,240)}, rotate = 0] [color={rgb, 255:red, 0; green, 0; blue, 0 }  ][fill={rgb, 255:red, 0; green, 0; blue, 0 }  ][line width=1.5]      (0, 0) circle [x radius= 4.36, y radius= 4.36]   ;
	\draw [line width=1.5]    (179.11,90) .. controls (288.91,108) and (270.91,111) .. (340,140) ;
	\draw [shift={(262.64,106.84)}, rotate = 195.73] [fill={rgb, 255:red, 0; green, 0; blue, 0 }  ][line width=0.08]  [draw opacity=0] (13.4,-6.43) -- (0,0) -- (13.4,6.44) -- (8.9,0) -- cycle    ;
	\draw [shift={(179.11,90)}, rotate = 9.31] [color={rgb, 255:red, 0; green, 0; blue, 0 }  ][fill={rgb, 255:red, 0; green, 0; blue, 0 }  ][line width=1.5]      (0, 0) circle [x radius= 4.36, y radius= 4.36]   ;
	\draw [line width=1.5]    (338.49,140) .. controls (422.4,174) and (433.2,174) .. (591.6,190) ;
	\draw [shift={(591.6,190)}, rotate = 5.77] [color={rgb, 255:red, 0; green, 0; blue, 0 }  ][fill={rgb, 255:red, 0; green, 0; blue, 0 }  ][line width=1.5]      (0, 0) circle [x radius= 4.36, y radius= 4.36]   ;
	\draw [shift={(462.85,175.95)}, rotate = 8.28] [fill={rgb, 255:red, 0; green, 0; blue, 0 }  ][line width=0.08]  [draw opacity=0] (13.4,-6.43) -- (0,0) -- (13.4,6.44) -- (8.9,0) -- cycle    ;
	\draw [line width=1.5]    (338.49,240) .. controls (422.4,274) and (433.2,274) .. (591.6,290) ;
	\draw [shift={(591.6,290)}, rotate = 5.77] [color={rgb, 255:red, 0; green, 0; blue, 0 }  ][fill={rgb, 255:red, 0; green, 0; blue, 0 }  ][line width=1.5]      (0, 0) circle [x radius= 4.36, y radius= 4.36]   ;
	\draw [shift={(462.85,275.95)}, rotate = 8.28] [fill={rgb, 255:red, 0; green, 0; blue, 0 }  ][line width=0.08]  [draw opacity=0] (13.4,-6.43) -- (0,0) -- (13.4,6.44) -- (8.9,0) -- cycle    ;
	\draw [line width=1.5]    (177.6,190) .. controls (240.6,205) and (280.2,211) .. (338.49,240) ;
	\draw [shift={(259.81,209.75)}, rotate = 195.33] [fill={rgb, 255:red, 0; green, 0; blue, 0 }  ][line width=0.08]  [draw opacity=0] (13.4,-6.43) -- (0,0) -- (13.4,6.44) -- (8.9,0) -- cycle    ;
	\draw [shift={(177.6,190)}, rotate = 13.39] [color={rgb, 255:red, 0; green, 0; blue, 0 }  ][fill={rgb, 255:red, 0; green, 0; blue, 0 }  ][line width=1.5]      (0, 0) circle [x radius= 4.36, y radius= 4.36]   ;
	\draw [line width=1.5]    (338.49,140) .. controls (262.2,112) and (465.6,280) .. (338.49,240) ;
	\draw [shift={(361.31,199.51)}, rotate = 53.19] [fill={rgb, 255:red, 0; green, 0; blue, 0 }  ][line width=0.08]  [draw opacity=0] (13.4,-6.43) -- (0,0) -- (13.4,6.44) -- (8.9,0) -- cycle    ;
	\draw [line width=1.5]    (338.49,240) .. controls (229.8,168) and (219,89) .. (338.49,140) ;
	\draw [shift={(257.12,159.01)}, rotate = 423.46000000000004] [fill={rgb, 255:red, 0; green, 0; blue, 0 }  ][line width=0.08]  [draw opacity=0] (13.4,-6.43) -- (0,0) -- (13.4,6.44) -- (8.9,0) -- cycle    ;
	\draw [line width=1.5]    (338.49,140) .. controls (456.6,181) and (559.2,334) .. (338.49,240) ;
	\draw [shift={(463.81,236.75)}, rotate = 61.62] [fill={rgb, 255:red, 0; green, 0; blue, 0 }  ][line width=0.08]  [draw opacity=0] (13.4,-6.43) -- (0,0) -- (13.4,6.44) -- (8.9,0) -- cycle    ;
	\draw [line width=1.5]    (338.49,40) .. controls (422.4,74) and (433.2,74) .. (591.6,90) ;
	\draw [shift={(462.85,75.95)}, rotate = 8.28] [fill={rgb, 255:red, 0; green, 0; blue, 0 }  ][line width=0.08]  [draw opacity=0] (13.4,-6.43) -- (0,0) -- (13.4,6.44) -- (8.9,0) -- cycle    ;
	\draw [line width=1.5]    (180,290) ;
	\draw [shift={(180,290)}, rotate = 0] [color={rgb, 255:red, 0; green, 0; blue, 0 }  ][fill={rgb, 255:red, 0; green, 0; blue, 0 }  ][line width=1.5]      (0, 0) circle [x radius= 4.36, y radius= 4.36]   ;
	
	\draw (128,132) node [anchor=north west][inner sep=0.75pt]   [align=left] {0};
	\draw (148.2,12) node [anchor=north west][inner sep=0.75pt]   [align=left] {$\displaystyle \alpha $};
	\draw (544.47,142) node [anchor=north west][inner sep=0.75pt]   [align=left] {$\displaystyle \theta $};
	\draw (106,82) node [anchor=north west][inner sep=0.75pt]   [align=left] {$\displaystyle \pi /2$};
	\draw (126,32) node [anchor=north west][inner sep=0.75pt]   [align=left] {$\displaystyle \pi $};
	\draw (95,180) node [anchor=north west][inner sep=0.75pt]   [align=left] {$\displaystyle -\pi /2$};
	\draw (116,230) node [anchor=north west][inner sep=0.75pt]   [align=left] {$\displaystyle -\pi $};
	\draw (204.02,72) node [anchor=north west][inner sep=0.75pt]   [align=left] {$\displaystyle S_{1}$};
	\draw (548.88,160) node [anchor=north west][inner sep=0.75pt]   [align=left] {$\displaystyle S_{3}$};
	\draw (204.02,170) node [anchor=north west][inner sep=0.75pt]   [align=left] {$\displaystyle S_{2}$};
	\draw (548.88,262) node [anchor=north west][inner sep=0.75pt]   [align=left] {$\displaystyle S_{4}$};
	\draw (346.56,220) node [anchor=north west][inner sep=0.75pt]   [align=left] {$\displaystyle O_{+}$};
	\draw (351,20) node [anchor=north west][inner sep=0.75pt]   [align=left] {$\displaystyle O_{-}$};
	\draw (293.2,92) node [anchor=north west][inner sep=0.75pt]   [align=left] {$\displaystyle \gamma _{1}$};
	\draw (493,152) node [anchor=north west][inner sep=0.75pt]   [align=left] {$\displaystyle \gamma _{2}$};
	\draw (241,180) node [anchor=north west][inner sep=0.75pt]   [align=left] {$\displaystyle \gamma _{3}$};
	\draw (491.2,260) node [anchor=north west][inner sep=0.75pt]   [align=left] {$\displaystyle \gamma _{4}$};
	\draw (87,280) node [anchor=north west][inner sep=0.75pt]   [align=left] {$\displaystyle -3\pi /2$};

	\end{tikzpicture}	
	\caption{Typical orbits when $n\geq4g$.}
	\end{center}
\end{figure}

\begin{theorem}[Orbit Types B]
	\label{orbitB}
	If $n\geq 4g$ and $(g,m_1,m_2)$ is stable, then any orbit $\gamma$ of the vector field in $0 < \theta < \frac{\pi}{2}$  is of one of the following types:
	\begin{itemize}
		\item Fixed  : $\gamma$ is one of the fixed points $O_+$ or $O_-$.
		\item Type I : $\gamma$ is a separatrix with $L^{\pm}(\gamma) = O_{\mp}$ and $L^{\mp}(\gamma) = S_i$, decreasing/increasing monotonically to $O_\pm$.
		\item Type II :
		\begin{itemize}
		\item [a)] $\gamma$ exits $O_+$ to the right and enters $O_-$ from the left, crossing both the line $\alpha = 0$ and $\alpha = -\pi$.
		\item [b)] $\gamma$ exits $O_+$ to the left and enters $O_-$ from the left or $\gamma$ exits $O_+$ to the right and enters $O_-$ from the right, crossing $\gamma$ $\alpha = 0$ or $\alpha = -\pi$ respectively.
		\end{itemize}
	
	\end{itemize}
\end{theorem}

We now discuss how to determine for which triples $(g,m_1,m_2)$ the separatrices $\gamma_1$ and $\gamma_3$ cross the vertical line $\theta=\theta^*$. Recall, we call these triples stable. We already know they always cross whenever $n<4g$ so we only need to consider $n\geq 4g$. Further, the map $(\theta,\alpha)\to (\frac{\pi}{2}-\theta,-\alpha)$ maps the vector field generated by $(g,m_1,m_2)$ to the one generated by $(g,m_2,m_1)$ and the separatrix $\gamma_1$ to $\gamma_3$, thus it suffices to consider the behaviour of $\gamma_1$.
One proceeds by constructing curves from $S_1$ to $O_-$ which $\gamma_1$ does not cross. Since no orbit can cross from above the $\alpha=0$ line to below it before $\theta^*$ such a boundary would show that $\gamma_1$ goes directly to $(\theta^*,0)$ and does not cross the line $\theta=\theta^*$. One considers a family of potential boundaries $B_a$ and reduces the question of whether they are boundaries to a question about the positivity of a certain polynomial on an interval. 
The results are summarised in the following theorem, for details we refer to \cite{Wang}.
\begin{theorem}
	For $n\geq4g$, $\gamma_1$ only crosses the $\theta=\theta^*$ for the triples $(2,1,5)$ and $(4,1,6)$. Consequently, all triples with $n\geq4g$ and $(g,m_1,m_2)\neq (2,1,5),(4,1,6)$ are stable.
\end{theorem}

\begin{figure}

\begin{center}
\tikzset{every picture/.style={line width=0.75pt}} 

\begin{tikzpicture}[x=0.75pt,y=0.75pt,yscale=-1,xscale=1]

\draw [line width=1.5]    (161.35,42) -- (161.35,103) -- (160,190) ;
\draw [line width=1.5]    (160.68,146.5) -- (398.68,151.5) ;
\draw [shift={(160.68,146.5)}, rotate = 1.2] [color={rgb, 255:red, 0; green, 0; blue, 0 }  ][fill={rgb, 255:red, 0; green, 0; blue, 0 }  ][line width=1.5]      (0, 0) circle [x radius= 4.36, y radius= 4.36]   ;
\draw [line width=1.5]    (161.35,103) ;
\draw    (160,60) .. controls (273.5,96) and (311.5,126) .. (350,150) ;
\draw [shift={(350,150)}, rotate = 31.94] [color={rgb, 255:red, 0; green, 0; blue, 0 }  ][fill={rgb, 255:red, 0; green, 0; blue, 0 }  ][line width=0.75]      (0, 0) circle [x radius= 3.35, y radius= 3.35]   ;
\draw [shift={(160,60)}, rotate = 17.6] [color={rgb, 255:red, 0; green, 0; blue, 0 }  ][fill={rgb, 255:red, 0; green, 0; blue, 0 }  ][line width=0.75]      (0, 0) circle [x radius= 3.35, y radius= 3.35]   ;
\draw  [dash pattern={on 0.84pt off 2.51pt}]  (160,60) .. controls (251.5,71) and (328.5,109) .. (350,150) ;

\draw (138,132) node [anchor=north west][inner sep=0.75pt]   [align=left] {0};
\draw (116,52) node [anchor=north west][inner sep=0.75pt]   [align=left] {$\displaystyle \pi /2$};
\draw (161,42) node [anchor=north west][inner sep=0.75pt]   [align=left] {$\displaystyle S_{1}$};
\draw (355,122) node [anchor=north west][inner sep=0.75pt]   [align=left] {$\displaystyle O_{-}$};
\draw (230,90) node [anchor=north west][inner sep=0.75pt]   [align=left] {$\displaystyle \gamma _{1}$};
\draw (270,62) node [anchor=north west][inner sep=0.75pt]   [align=left] {$\displaystyle B_{a}$};

\end{tikzpicture}
	\caption{The barrier $B_a$.}
\end{center}

\end{figure}
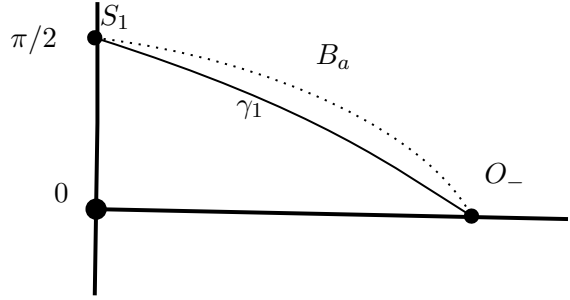
\section{Profile curves}
\label{profile}
We now use our knowledge about orbits of the vector field to classify profile curves and the corresponding minimal surfaces. 

\smallskip

Recall a profile curve $\Gamma = (r(t),\theta(t),\alpha(t))$ corresponding to an orbit $\gamma = (\theta(t),\alpha(t))$ is a curve where $r(t)$ solves
\[\frac{dr}{dt}=r\sin 2\theta \cos(\alpha) .\]
Since $r$ is positive and $ \sin 2\theta$ are positive, $r$ has a critical point only when $\alpha = (k+1/2)\pi, k\in \mathbb{Z}$. These points are always minima.

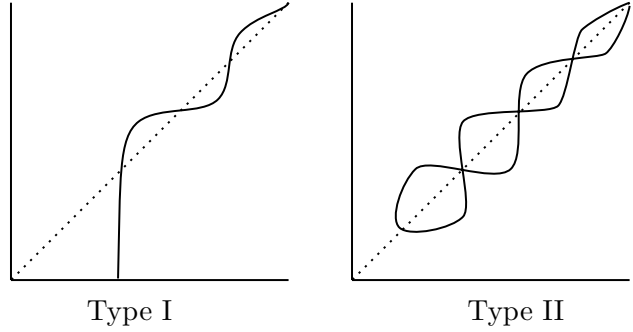
\begin{figure}[ht]
\begin{center}

	\tikzset{every picture/.style={line width=0.75pt}} 
	
	\begin{tikzpicture}[x=0.75pt,y=0.75pt,yscale=-1,xscale=1,scale = 0.8]
	
	\draw    (340,105) -- (340,280) ;
	\draw    (515,280) -- (340,280) ;
	\draw  [dash pattern={on 0.84pt off 2.51pt}]  (515,105) -- (340,280) ;
	\draw    (125,105) -- (125,280) ;
	\draw    (300,280) -- (125,280) ;
	\draw  [dash pattern={on 0.84pt off 2.51pt}]  (300,105) -- (125,280) ;
	\draw    (192.5,279) .. controls (194,215.5) and (192.5,190) .. (205.5,180) .. controls (218.5,170) and (245.5,176) .. (255.5,166) .. controls (265.5,156) and (259.5,135) .. (269.5,124) .. controls (279.5,113) and (298.5,110) .. (300,105) ;
	\draw    (515,105) .. controls (505.6,108.2) and (489.8,117.4) .. (485,122) .. controls (480.2,126.6) and (475.9,164) .. (470,170) .. controls (464.1,176) and (417.8,170.8) .. (410,180) .. controls (402.2,189.2) and (417,232.4) .. (410,240) .. controls (403,247.6) and (376.6,253) .. (370,247) .. controls (363.4,241) and (371.4,218.4) .. (380,210) .. controls (388.6,201.6) and (429.8,219.6) .. (440,210) .. controls (450.2,200.4) and (440,161.2) .. (450,150) .. controls (460,138.8) and (495,140.6) .. (500,137) .. controls (505,133.4) and (514,115) .. (515,107) ;
	
	\draw (411,291) node [anchor=north west][inner sep=0.75pt]   [align=left] {Type II};
	\draw (171,291) node [anchor=north west][inner sep=0.75pt]   [align=left] {Type I};

	\end{tikzpicture}
	
\caption{Profile curves for $n<4g$.}
\end{center}	
\end{figure}
\begin{theorem}[ Profile curves A]
	\label{ProfileA}
	Any profile curve $\Gamma$ corresponding to a F-invariant minimal surface in $\R^n$ with a triple with $n<4g$, is of one of the following types:
	\begin{itemize}
		\item Ray : $\Gamma$ is the ray $\Gamma^*:\theta = \theta^*$. 
		\item Type I : One end of $\Gamma$ intersects either the x-axis or the y-axis orthogonally and the other end is asymptotic to $\Gamma^*$. Further, $r(s)$ increases monotonically and $\Gamma$ intersects $\Gamma^*$ infinitely often.
		\item Type II : Both ends of $\Gamma$ are asymptotic to $\Gamma^*$ and $r(s)$ has a unique minima. Further, $\Gamma$ intersects $\Gamma^*$ and itself infinitely many times. 
	\end{itemize}
\end{theorem}

\begin{proof}
	In the fixed case \[\theta=\theta^* , \alpha=0\] for all time. So, 
	\[\frac{dr}{d\theta}= Cr\]
	where $C$ is a positive constant, thus we get a ray.\\
	
	For a Type I orbit, 
	\[\lim_{t\to-\infty}\theta(t) = 0  \quad (\text{ or } \frac{\pi}{2}) \text{ and }\lim_{t\to-\infty}\alpha(t) = \frac{\pi}{2}  \quad (\text{ or } -\frac{\pi}{2}),\]
	thus the resulting profile curve will be perpendicular to the $x$ or $y$ axis. Further,
	\[\lim_{t\to\infty}\theta(t) = \theta^*\]
	and $\theta$ crosses the $\theta=\theta^*$ line infinitely often.
	Further, $-\frac{\pi}{2}< \alpha < \frac{\pi}{2}$, thus $\frac{dr}{d\theta}>0$ and approaches $Cr$,thus the profile curve approaches the ray $\Gamma^*$ as $t\to \infty$, intersecting it infinitely often.
	
	For a Type II orbit, 
	\[\lim_{t\to-\infty}\theta(t) = \theta^* \text{ and }\lim_{t\to-\infty}\alpha(t) = -\pi\]
	and 
	\[\lim_{t\to\infty}\theta(t) = \theta^* \text{ and }\lim_{t\to\infty}\alpha(t) = 0.\]
	Further, the orbit crosses the $\theta=\theta^*$ line infinitely often in both directions.	
	On the line $\alpha = -\frac{\pi}{2}$, we have $\frac{d\alpha}{dt}=m \sin(2\theta)>0$ thus no curve can cross this line from above. Further, for a profile curve, $r$ is stationary precisely when the corresponding orbit crosses this line. Thus, on any Type II profile curve there is a unique point where $r$ is stationary. As 
	\[\lim_{t\to-\infty} \frac{dr}{dt}<0 \text{ and }\lim_{t\to\infty} \frac{dr}{dt}>0,\]
	we have 
	\[\lim_{t\to\pm\infty} r(t)=\infty,\]
	thus the profile curve is doubly asymptotic to $\Gamma^*$, crossing it infinitely many times in both directions.
\end{proof}

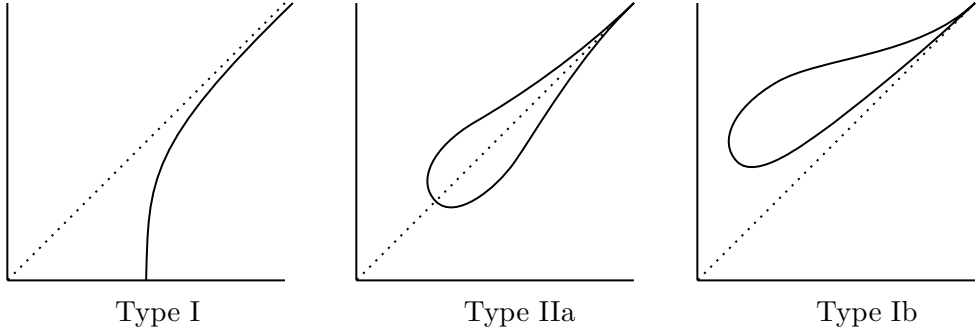
\begin{figure}[ht]
\begin{center}
		
\tikzset{every picture/.style={line width=0.75pt}} 

\begin{tikzpicture}[x=0.75pt,y=0.75pt,yscale=-1,xscale=1,scale = 0.8]

\draw    (260,105) -- (260,280) ;
\draw    (435,280) -- (260,280) ;
\draw  [dash pattern={on 0.84pt off 2.51pt}]  (435,105) -- (260,280) ;
\draw    (435,105) .. controls (399.5,139) and (362.5,164) .. (335,180) .. controls (307.5,196) and (297.75,217.5) .. (310,230) .. controls (322.25,242.5) and (347.5,223) .. (360,205) .. controls (372.5,187) and (400.5,138) .. (435,105) -- cycle ;
\draw    (475,105) -- (475,280) ;
\draw    (650,280) -- (475,280) ;
\draw  [dash pattern={on 0.84pt off 2.51pt}]  (650,105) -- (475,280) ;
\draw    (650,105) .. controls (614.5,139) and (552.5,139) .. (525,155) .. controls (497.5,171) and (487.75,192.5) .. (500,205) .. controls (512.25,217.5) and (543,193.5) .. (555,185) .. controls (567,176.5) and (615.5,138) .. (650,105) -- cycle ;
\draw    (40,105) -- (40,280) ;
\draw    (215,280) -- (40,280) ;
\draw  [dash pattern={on 0.84pt off 2.51pt}]  (215,105) -- (40,280) ;
\draw    (127.5,280) .. controls (129,216.5) and (128,194.5) .. (220,105) ;

\draw (106,291) node [anchor=north west][inner sep=0.75pt]   [align=left] {Type I};
\draw (326,291) node [anchor=north west][inner sep=0.75pt]   [align=left] {Type IIa};
\draw (548,291) node [anchor=north west][inner sep=0.75pt]   [align=left] {Type Ib};

\end{tikzpicture}

\caption{Profile curves for $n\geq 4g$.}
\end{center}
\end{figure}
\begin{theorem}[Profile curves B1]
	\label{ProfileB}
	Any profile curve $\Gamma$, corresponding to a F-invariant minimal surface in $\R^n$ with a stable triple $(g,m_1,m_2)$ (not $(2,1,5)$,$(4,1,6)$) with $n\geq4g$, is of one of the following types:
	
	\begin{itemize}

		\item Ray : $\Gamma$ is the ray $\Gamma^*:\theta = \theta^*$. 
		
		\item Type I : One end of $\Gamma$ intersects either the x-axis or the y-axis orthogonally and the other end is asymptotic to $\Gamma^*$. Further, $r(s)$ increases monotonically and $\Gamma$ never intersects $\Gamma^*$.
		
		\item Type II :  Both ends of $\Gamma$ are asymptotic to $\Gamma^*$ and $r(s)$ has a unique minima. Further, $\Gamma$ intersects $\Gamma^*$ at most once. Further, $\Gamma$ never intersects itself.
	\end{itemize}
\end{theorem}
\begin{proof}
	If the orbit is fixed, the proof is identical to above.\\
	
	If the orbit is Type I, the proof is similar to above except since the orbit does not cross the line $\theta=\theta^*$, it never intersects the ray $\Gamma^*$.\\
	
	If $\gamma$ is Type II, the proof that $\Gamma$ is asymptotic to $\Gamma^*$ and has a unqiue minima still holds. Further, it only intersects $\Gamma^*$ at most once, since $\frac{d\theta}{dt}<0$ for $-\pi<\alpha<0$ and the curves $\gamma_1,\gamma_4$ prevent $\gamma$ from crossing the $\theta=\theta^*$ line again.\\
	
	To show $\Gamma$ never intersects itself, we proceed as follows. If $\Gamma$ intersects $\Gamma^*$ we can split it into two sections, $\Gamma^+$ and $\Gamma^-$, such that both lie strictly on one side of $\Gamma^*$. Assume w.l.o.g that $\Gamma\subset \{(r,\theta): \theta<\theta^*, (r,\theta)\in D_g\}$. Fix one profile curve corresponding to $\gamma_1$. We can write this curve as $r_1(\theta)$, $0<\theta<\theta^*$. Consider the function 
	\[f(r,\theta) = \frac{r}{r_1(\theta)}.\]
	The level sets of this function are exactly the Type I profile curves corresponding to $\gamma_1$. Assume $\Gamma(t_0)=\Gamma(t_1)$. Then by Rolle's theorem there exits a point $x_0$ on $\Gamma$, such that $\Gamma$ is tangent to a level set of $f$ at $x_0$. This is a contradiction since $\Gamma$ is of Type II but the level sets of $f$ are Type I and since profile curves are geodesics, this violates the uniqueness of geodesics.	
\end{proof}

\begin{theorem}[Profile curves B2]
	\label{ProfileC}
	Any profile curve $\Gamma$, corresponding to a F-invariant minimal surface in $\R^n$ with an unstable triple $(g,m_1,m_2)$ ( $(2,1,5)$,$(4,1,6)$) with $n\geq4g$, is of one of the following types:
	
	\begin{itemize}
		
		\item Ray : $\Gamma$ is the ray $\Gamma^*:\theta = \theta^*$. 
		
		\item Type I : One end of $\Gamma$ intersects either the x-axis or the y-axis orthogonally and the other end is asymptotic to $\Gamma^*$. Further, $r(s)$ increases monotonically and $\Gamma$ intersects $\Gamma^*$ at most once.
		
		\item Type II :  Both ends of $\Gamma$ are asymptotic to $\Gamma^*$ and $r(s)$ has a unique minima. Further, $\Gamma$ intersects $\Gamma^*$ at most once.
	\end{itemize}
\end{theorem}

\begin{proof}
	The proof is identical to the stable case, except in the following cases.
	
	In the Type I case, since one $\gamma_i$ can cross the line $\theta = \theta^*$ once, $\Gamma$ can intersect the ray $\Gamma^*$. In the Type II case, the argument for embeddedness fails.
\end{proof}

\section{Free Boundary Minimal Surfaces}
\label{fbms}
We can now use the results from previous sections to construct F-invariant free boundary minimal surfaces in the unit ball. We say a smooth submanifold $\Sigma\subset B^k(1)$ is a \textit{free boundary minimal surface} in $B^k(1)$ if $\Sigma$ is minimal, $\partial\Sigma\subset B^k(1)$, and $\Sigma$ intersects $\partial B^k(1)$ orthogonally along $\partial\Sigma$.
In Section 5 and 6 we proved a reduction theorem, showing that to find F-invariant minimal surfaces in $\R^n$, it suffices to study so called profile curves. The equations describing these profile curves just depend on the triple $(g,m_1,m_2)$, where $g$ is the number of distinct principal curvatures associated to the isoparametric hypersurface corresponding to $F$ and $m_1$,$m_2$ are the two multiplicities of these curvatures. The condition that $\Sigma$ intersects $\partial B^k(1)$ orthogonally along $\partial\Sigma$ immediately translates to the same condition on profile curves. Thus, we say a profile curve $\gamma$ in $D_g$ is a \textit{free boundary profile curve} if for all $t$ such that $\norm{\gamma(t)} = 1$, $\gamma(t)=\gamma'(t)$. In this context we construct two families of free boundary minimal surfaces:  $\Sigma_{g,m_1,m_2}^k$ and $\Omega_{g,m_1,m_2}$. $\Sigma_{g,m_1,m_2}^k$ is constructed in \cite{Wang} and we simply repeat its construction. A special case also appears in \cite{McGrath}. The family $\Omega_{g,m_1,m_2}$ is based on a construction from \cite{McGrath} but, to our knowledge, is completely new in this generality. A full list of triples satisfying the conditions is provided in Section \ref{Isoparametric hypersurfaces}.

\begin{figure}
	\label{freeboundaryfig}
	\begin{center}

	\tikzset{every picture/.style={line width=0.75pt}} 
	
	\begin{tikzpicture}[x=0.75pt,y=0.75pt,yscale=-1,xscale=1,scale=0.8]
	
	\draw [line width=1.5]    (66.24,39.38) -- (66.95,78.29) -- (65.34,275.21) ;
	\draw [line width=1.5]    (313.5,42.91) -- (312.37,275.21) ;
	\draw [line width=1.5]  [dash pattern={on 1.69pt off 2.76pt}]  (65.95,96.33) -- (311.37,96.33) ;
	\draw [shift={(311.37,96.33)}, rotate = 0] [color={rgb, 255:red, 0; green, 0; blue, 0 }  ][fill={rgb, 255:red, 0; green, 0; blue, 0 }  ][line width=1.5]      (0, 0) circle [x radius= 3.05, y radius= 3.05]   ;
	\draw [shift={(65.95,96.33)}, rotate = 0] [color={rgb, 255:red, 0; green, 0; blue, 0 }  ][fill={rgb, 255:red, 0; green, 0; blue, 0 }  ][line width=1.5]      (0, 0) circle [x radius= 3.05, y radius= 3.05]   ;
	\draw [line width=1.5]  [dash pattern={on 1.69pt off 2.76pt}]  (66.95,216.25) -- (312.37,216.25) ;
	\draw [shift={(312.37,216.25)}, rotate = 0] [color={rgb, 255:red, 0; green, 0; blue, 0 }  ][fill={rgb, 255:red, 0; green, 0; blue, 0 }  ][line width=1.5]      (0, 0) circle [x radius= 3.05, y radius= 3.05]   ;
	\draw [shift={(66.95,216.25)}, rotate = 0] [color={rgb, 255:red, 0; green, 0; blue, 0 }  ][fill={rgb, 255:red, 0; green, 0; blue, 0 }  ][line width=1.5]      (0, 0) circle [x radius= 3.05, y radius= 3.05]   ;
	\draw [line width=1.5]    (66.95,78.29) ;
	\draw [line width=1.5]    (161.34,98.33) ;
	\draw [shift={(161.34,98.33)}, rotate = 0] [color={rgb, 255:red, 0; green, 0; blue, 0 }  ][fill={rgb, 255:red, 0; green, 0; blue, 0 }  ][line width=1.5]      (0, 0) circle [x radius= 4.36, y radius= 4.36]   ;
	\draw [line width=1.5]    (161.34,216.25) ;
	\draw [shift={(161.34,216.25)}, rotate = 0] [color={rgb, 255:red, 0; green, 0; blue, 0 }  ][fill={rgb, 255:red, 0; green, 0; blue, 0 }  ][line width=1.5]      (0, 0) circle [x radius= 4.36, y radius= 4.36]   ;
	\draw [line width=1.5]    (65.24,39.38) .. controls (130.75,60.6) and (120.01,64.14) .. (161.24,98.33) ;
	\draw [shift={(118.35,61.24)}, rotate = 210.88] [fill={rgb, 255:red, 0; green, 0; blue, 0 }  ][line width=0.08]  [draw opacity=0] (9.91,-4.76) -- (0,0) -- (9.91,4.76) -- (6.58,0) -- cycle    ;
	\draw [shift={(65.24,39.38)}, rotate = 17.95] [color={rgb, 255:red, 0; green, 0; blue, 0 }  ][fill={rgb, 255:red, 0; green, 0; blue, 0 }  ][line width=1.5]      (0, 0) circle [x radius= 3.05, y radius= 3.05]   ;
	\draw [line width=1.5]    (161.34,98.33) .. controls (211.41,138.42) and (217.85,138.42) .. (312.37,157.29) ;
	\draw [shift={(312.37,157.29)}, rotate = 11.29] [color={rgb, 255:red, 0; green, 0; blue, 0 }  ][fill={rgb, 255:red, 0; green, 0; blue, 0 }  ][line width=1.5]      (0, 0) circle [x radius= 3.05, y radius= 3.05]   ;
	\draw [shift={(232.12,139.74)}, rotate = 17.34] [fill={rgb, 255:red, 0; green, 0; blue, 0 }  ][line width=0.08]  [draw opacity=0] (9.91,-4.76) -- (0,0) -- (9.91,4.76) -- (6.58,0) -- cycle    ;
	\draw [line width=1.5]    (161.34,216.25) .. controls (211.41,256.34) and (217.85,256.34) .. (312.37,275.21) ;
	\draw [shift={(312.37,275.21)}, rotate = 11.29] [color={rgb, 255:red, 0; green, 0; blue, 0 }  ][fill={rgb, 255:red, 0; green, 0; blue, 0 }  ][line width=1.5]      (0, 0) circle [x radius= 4.36, y radius= 4.36]   ;
	\draw [shift={(232.12,257.65)}, rotate = 197.34] [fill={rgb, 255:red, 0; green, 0; blue, 0 }  ][line width=0.08]  [draw opacity=0] (13.4,-6.43) -- (0,0) -- (13.4,6.44) -- (8.9,0) -- cycle    ;
	\draw [line width=1.5]    (65.34,157.29) .. controls (102.93,174.98) and (126.56,182.05) .. (161.34,216.25) ;
	\draw [shift={(116.64,181.83)}, rotate = 28.78] [fill={rgb, 255:red, 0; green, 0; blue, 0 }  ][line width=0.08]  [draw opacity=0] (9.91,-4.76) -- (0,0) -- (9.91,4.76) -- (6.58,0) -- cycle    ;
	\draw [shift={(65.34,157.29)}, rotate = 25.2] [color={rgb, 255:red, 0; green, 0; blue, 0 }  ][fill={rgb, 255:red, 0; green, 0; blue, 0 }  ][line width=1.5]      (0, 0) circle [x radius= 3.05, y radius= 3.05]   ;
	\draw [line width=1.5]    (66.77,275.21) ;
	\draw [shift={(66.77,275.21)}, rotate = 0] [color={rgb, 255:red, 0; green, 0; blue, 0 }  ][fill={rgb, 255:red, 0; green, 0; blue, 0 }  ][line width=1.5]      (0, 0) circle [x radius= 3.05, y radius= 3.05]   ;
	\draw [line width=1.5]  [dash pattern={on 1.69pt off 2.76pt}]  (65.95,157.29) -- (311.37,157.29) ;
	\draw [shift={(311.37,157.29)}, rotate = 0] [color={rgb, 255:red, 0; green, 0; blue, 0 }  ][fill={rgb, 255:red, 0; green, 0; blue, 0 }  ][line width=1.5]      (0, 0) circle [x radius= 3.05, y radius= 3.05]   ;
	\draw [shift={(65.95,157.29)}, rotate = 0] [color={rgb, 255:red, 0; green, 0; blue, 0 }  ][fill={rgb, 255:red, 0; green, 0; blue, 0 }  ][line width=1.5]      (0, 0) circle [x radius= 3.05, y radius= 3.05]   ;
	\draw    (161.34,216.25) .. controls (198.79,244) and (248.72,241) .. (248.12,217) .. controls (247.53,193) and (221.31,170) .. (209.5,160) .. controls (203.11,154.59) and (179.08,143.61) .. (158.29,131.03) .. controls (140.67,120.36) and (125.37,108.55) .. (125.09,98) .. controls (124.5,75) and (148.87,88) .. (161.34,98.33) ;
	\draw [shift={(207.68,235.34)}, rotate = 189.45] [fill={rgb, 255:red, 0; green, 0; blue, 0 }  ][line width=0.08]  [draw opacity=0] (10.72,-5.15) -- (0,0) -- (10.72,5.15) -- (7.12,0) -- cycle    ;
	\draw [shift={(234.22,184.76)}, rotate = 413.75] [fill={rgb, 255:red, 0; green, 0; blue, 0 }  ][line width=0.08]  [draw opacity=0] (10.72,-5.15) -- (0,0) -- (10.72,5.15) -- (7.12,0) -- cycle    ;
	\draw [shift={(183.76,145.26)}, rotate = 387.97] [fill={rgb, 255:red, 0; green, 0; blue, 0 }  ][line width=0.08]  [draw opacity=0] (10.72,-5.15) -- (0,0) -- (10.72,5.15) -- (7.12,0) -- cycle    ;
	\draw [shift={(138.25,117.16)}, rotate = 398.1] [fill={rgb, 255:red, 0; green, 0; blue, 0 }  ][line width=0.08]  [draw opacity=0] (10.72,-5.15) -- (0,0) -- (10.72,5.15) -- (7.12,0) -- cycle    ;
	\draw [shift={(140.63,86.29)}, rotate = 190.22] [fill={rgb, 255:red, 0; green, 0; blue, 0 }  ][line width=0.08]  [draw opacity=0] (10.72,-5.15) -- (0,0) -- (10.72,5.15) -- (7.12,0) -- cycle    ;
	\draw    (330,159) -- (382.5,158.04) ;
	\draw [shift={(384.5,158)}, rotate = 538.95] [color={rgb, 255:red, 0; green, 0; blue, 0 }  ][line width=0.75]    (10.93,-3.29) .. controls (6.95,-1.4) and (3.31,-0.3) .. (0,0) .. controls (3.31,0.3) and (6.95,1.4) .. (10.93,3.29)   ;
	\draw [line width=1.5]    (125.09,98) ;
	\draw [shift={(125.09,98)}, rotate = 0] [color={rgb, 255:red, 0; green, 0; blue, 0 }  ][fill={rgb, 255:red, 0; green, 0; blue, 0 }  ][line width=1.5]      (0, 0) circle [x radius= 3.05, y radius= 3.05]   ;
	\draw [line width=1.5]    (248.12,217) ;
	\draw [shift={(248.12,217)}, rotate = 0] [color={rgb, 255:red, 0; green, 0; blue, 0 }  ][fill={rgb, 255:red, 0; green, 0; blue, 0 }  ][line width=1.5]      (0, 0) circle [x radius= 3.05, y radius= 3.05]   ;
	\draw [line width=1.5]    (209.5,160) ;
	\draw [shift={(209.5,160)}, rotate = 0] [color={rgb, 255:red, 0; green, 0; blue, 0 }  ][fill={rgb, 255:red, 0; green, 0; blue, 0 }  ][line width=1.5]      (0, 0) circle [x radius= 3.05, y radius= 3.05]   ;
	\draw    (401.5,42) -- (401.5,280) ;
	\draw    (632.5,280) -- (401.5,280) ;
	\draw  [draw opacity=0][dash pattern={on 0.84pt off 2.51pt}] (404,99.34) .. controls (512.52,99.52) and (600.58,179.01) .. (600.91,277.35) .. controls (600.91,277.56) and (600.91,277.78) .. (600.91,278) -- (403.5,278) -- cycle ; \draw  [dash pattern={on 0.84pt off 2.51pt}] (404,99.34) .. controls (512.52,99.52) and (600.58,179.01) .. (600.91,277.35) .. controls (600.91,277.56) and (600.91,277.78) .. (600.91,278) ;
	\draw    (461.5,108) .. controls (451.5,140) and (427.5,224) .. (447.5,242) .. controls (467.5,260) and (549.5,184) .. (564.5,173) ;
	\draw [shift={(444.14,175.33)}, rotate = 101.12] [fill={rgb, 255:red, 0; green, 0; blue, 0 }  ][line width=0.08]  [draw opacity=0] (10.72,-5.15) -- (0,0) -- (10.72,5.15) -- (7.12,0) -- cycle    ;
	\draw [shift={(510.11,216.73)}, rotate = 323.99] [fill={rgb, 255:red, 0; green, 0; blue, 0 }  ][line width=0.08]  [draw opacity=0] (10.72,-5.15) -- (0,0) -- (10.72,5.15) -- (7.12,0) -- cycle    ;
	\draw  [line width=0.75]  (460.33,114.43) -- (465,115.75) -- (466.98,108.75) ;
	\draw  [line width=0.75]  (558.4,177.87) -- (562.5,183) -- (568.6,178.13) ;
	\draw [line width=1.5]    (158.29,131.03) ;
	\draw [shift={(158.29,131.03)}, rotate = 0] [color={rgb, 255:red, 0; green, 0; blue, 0 }  ][fill={rgb, 255:red, 0; green, 0; blue, 0 }  ][line width=1.5]      (0, 0) circle [x radius= 3.05, y radius= 3.05]   ;
	
	\draw (33.53,90.42) node [anchor=north west][inner sep=0.75pt]   [align=left] {0};
	\draw (54.92,7.97) node [anchor=north west][inner sep=0.75pt]   [align=left] {$\displaystyle \alpha $};
	\draw (15.96,31.55) node [anchor=north west][inner sep=0.75pt]   [align=left] {$\displaystyle \pi /2$};
	\draw (7.18,147.11) node [anchor=north west][inner sep=0.75pt]   [align=left] {$\displaystyle -\pi /2$};
	\draw (23.95,206.07) node [anchor=north west][inner sep=0.75pt]   [align=left] {$\displaystyle -\pi $};
	\draw (0.79,265.03) node [anchor=north west][inner sep=0.75pt]   [align=left] {$\displaystyle -3\pi /2$};
	\draw (86.34,77) node [anchor=north west][inner sep=0.75pt]  [font=\small] [align=left] {$\displaystyle \gamma ( t_{+})$};
	\draw (173.8,160) node [anchor=north west][inner sep=0.75pt]  [font=\small] [align=left] {$\displaystyle \gamma ( t_{0})$};
	\draw (251.09,196) node [anchor=north west][inner sep=0.75pt]  [font=\small] [align=left] {$\displaystyle \gamma ( t_{-})$};
	\draw (546.09,184) node [anchor=north west][inner sep=0.75pt]  [font=\small] [align=left] {$\displaystyle \Gamma ( t_{-})$};
	\draw (414.09,241) node [anchor=north west][inner sep=0.75pt]  [font=\small] [align=left] {$\displaystyle \Gamma ( t_{0})$};
	\draw (426.09,99) node [anchor=north west][inner sep=0.75pt]  [font=\small] [align=left] {$\displaystyle \Gamma ( t_{+})$};
	\draw (128.8,134) node [anchor=north west][inner sep=0.75pt]  [font=\small] [align=left] {$\displaystyle \gamma ( 0)$};

	\end{tikzpicture}
	\caption{The construction in Theorem \ref{freeboundary}.2.}
	\end{center}
\end{figure}
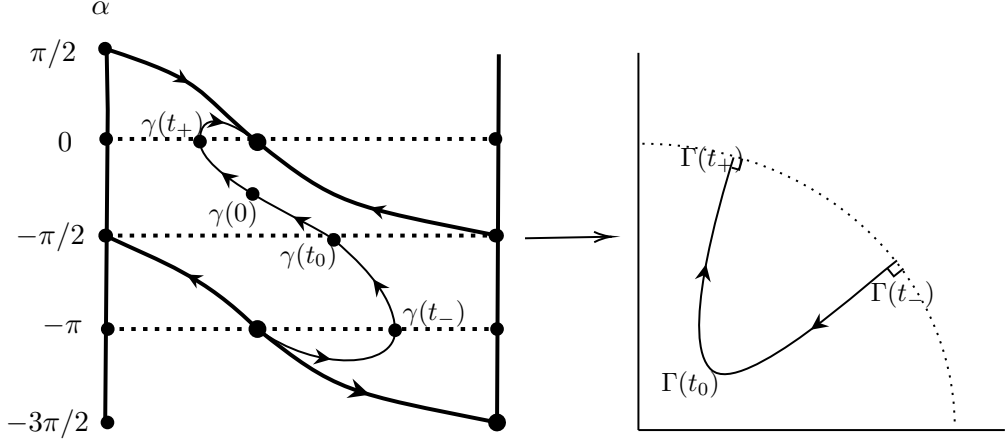
\begin{theorem}
	\label{freeboundary}
	Given an isoparametric hypersurface in $S^{n-1}$ with corresponding triple $(g,m_1,m_2)$ where $n=\frac{m_1+m_2}{2}g+2$, we can construct F-invariant free boundary minimal surfaces in $\R^n$ in the following ways:
	\begin{enumerate}
		\item 	If $n<4g$, for all natural numbers $k$ we can construct a free boundary minimal surface $\Sigma_{g,m_1,m_2}^k$.
		\item If $(g,m_1,m_2) \neq (2,1,5), (4,1,6)$ and $n\geq 4g$, we can construct a free boundary minimal surface $\Omega_{g,m_1,m_2}$.
	\end{enumerate}
\end{theorem}
\begin{proof}
	Case 1: Note that the condition $\gamma(t)=\gamma'(t)$ is equivalent to $\alpha(t)=k\pi$, $k\in\mathbb{Z}$. Consider the seperatrix $\gamma_1$. For $n<4g$, $O_-$ is focal so we know $\gamma_1$ crosses the line $\alpha=0$ infinitely many times.  Number the intersection times $t_k$. Consider the profile curve $\sigma^k(t)=(r(t),\theta(t),\alpha(t))$ in $D_g$ corresponding to $\gamma_1$ with $r(t_k)=1$. Since $\gamma_1$ stays within $-\frac{\pi}{2}\leq\alpha\leq\frac{\pi}{2}$, $r(t)$ is strictly increasing, so this is the only point with $r(t)=1$. Thus this is a free boundary profile curve. We can do this for all $k$, $(g,m_1,m_2)$ with $n<4g$, thus we can construct $\Sigma_{g,m_1,m_2}^k$.\\
	
	Case 2: Consider the family of profile curves $\Gamma_\epsilon= (r(t),\theta(t),\alpha(t))$ which satisfy
	\[r(0)=\frac{1}{2}\quad \theta(0)=\theta^*\quad \alpha(0)=-\epsilon.\]
	Consider the points $t_+,t_0$ and $t_-$, such that $\alpha(t_+)=0$, $\alpha(t_0)=-\frac{\pi}{2}$ and $\alpha(t_-)=-\pi$. These times exist since this these curves cannot be separatrices. We want to construct a curve $\Gamma$ such that $r(t_-)=r(t_+)$. We do this by demonstrating that for some $\epsilon$ between $0$ and $\pi$ this holds. \\
	
	For $\epsilon<\frac{\pi}{2}$,  $r(t)$ is increasing for $t\in (0,t_+)$, thus $r(t_+)>\frac{1}{2}$ as $\epsilon\to0$. By smooth dependence of ODE solutions on initial conditions, we have $\Gamma_\epsilon\to \Gamma^*$ as $\epsilon\to 0$. Thus, the minima of $\Gamma_\epsilon$ approaches 0, thus there exists a sequence $\epsilon_k\to0$ such that $r(t_0(\epsilon_k))\to0$. If we consider the Euclidean arc-length parametrization $s$ of $\Gamma$ we have
	\[\Delta r = r(s_-)-r(s_0)\leq \Delta s \leq \frac{\Delta \alpha}{\min_{s_-\leq s\leq s_0}\frac{d\alpha}{ds}}.\]
	We know
	\begin{equation}
	\label{summands}
		\frac{d\alpha}{ds}=\frac{1}{r}\left( -(n-1)\sin\alpha+H(\theta)\cos\alpha\right).
	\end{equation}
	There is a curve entering $O_-$ from the right which does not intersect the $\theta = \theta^*$ axis, thus $\theta(t)$ is bounded away from $\frac{\pi}{2}$ for all $\epsilon$. Thus 
	$\frac{d\alpha}{ds}$ is bounded.\\
	 
	As $\epsilon\to 0$, $\theta(t_0),\theta(t_-)\to \frac{\pi}{2}$ and for $\theta> \theta^*$ the two summands in \ref{summands} have the same sign, so if we show one of them is always bounded below by a constant, the expression in the brackets will be bounded away from $0$. Near the line $\alpha=-\frac{\pi}{2}$, we have $-(n-1)\sin\alpha >0 $ uniformly and elsewhere $H(\theta)\cos\alpha>0$ uniformly. Since $r>r(t_0)$, we know $\min_{s_-\leq s\leq s_0}\frac{d\alpha}{ds}\to \infty$ as $\epsilon\to 0$. This implies $\Delta r\to 0$ and further $r(t_-(\epsilon))\to 0$ so $r(t_-)<r(t_+)$ for some $\epsilon$. Repeating the same argument with $\epsilon\to \pi$, we get $r(t_-)>r(t_+)$ for some $\epsilon$. By continuity, there must exist an $\epsilon$ such that $r(t_-(\epsilon))=r(t_+(\epsilon))$. The curve $\Sigma=\Gamma_\epsilon$ is then a free boundary curve. We can do this for all $(g,m_1,m_2)$ with $n\geq4g$, thus we can construct $\Omega_{g,m_1,m_2}$.
\end{proof}


\end{document}